\documentclass{amsart}

\usepackage[T2A,T1]{fontenc}
\usepackage[utf8]{inputenc}
\usepackage[russian,english]{babel}

\usepackage{graphicx}
\usepackage{booktabs}
\usepackage{tikz}
\usepackage{comment}
\usepackage{overpic}
\usepackage[a4paper,margin=2.5cm]{geometry}
\usepackage{amsmath,amssymb}
\usepackage{bm}
\usepackage{amsthm} % For theorem environments
\usepackage{enumitem}

\newtheorem{thm}{Theorem}
\newtheorem{lem}[thm]{Lemma}

\theoremstyle{definition}
\newtheorem{dfn}[thm]{Definition}
\newtheorem{rem}[thm]{Remark}

\renewcommand{\phi}{\varphi}

\newcommand\RR{\mathbb{R}}
\newcommand\CC{\mathbb{C}}
\newcommand\ZZ{\mathbb{Z}}
\newcommand\Sph{\mathbb{S}}

\title{Rigidity and flexibility of discrete conjugate nets with flexible $3 \times 3$-subnets}
\author{Ivan Izmestiev}
\address{TU Wien, Wiedner Hauptstraße 8-10/104, A-1040 Wien, Austria}
\email{izmestiev@geometrie.tuwien.ac.at}
\thanks{This research was funded in whole or in part by the Austrian Science Fund (FWF) 10.55776/F77. For open access purposes, the authors have applied a CC BY public copyright license to any author accepted manuscript version arising from this submission.}
\author{Georg Nawratil}
\email{nawratil@geometrie.tuwien.ac.at}
\author{Sobhan Samareh Rad}
\email{sobhan.rad@tuwien.ac.at}
\author{Kiumars Sharifmoghaddam}
\email{ksharif@geometrie.tuwien.ac.at}
\date{\today}

\begin{document}

\begin{abstract}
Discrete conjugate nets (also known as quad-surfaces) are polyhedral surfaces made of quadrilaterals connected in the combinatorics of the square grid.
A generic discrete conjugate net is rigid.
The present article contains counterexamples to an erroneous statement that a non-degenerate discrete conjugate net with flexible $3 \times 3$-subnets is flexible and proves a corrected version of it under stronger non-degeneracy assumptions.
\end{abstract}

\maketitle

\section{Introduction}
A \emph{discrete conjugate net} is a polyhedral surface with quadrilateral faces that are connected in the combinatorics of the square grid.
Other terms used for an object of this kind are planar-quad surface (PQ-surface) or just quad-surface or PQ-net.
An \emph{isometric deformation} of a discrete conjugate net is a one-parameter smooth family of discrete conjugate nets where every face undergoes a rigid motion in space.
An isometric deformation where all faces ungergo the same rigid motion is called \emph{trivial}.
A discrete conjugate net that admits a non-trivial isometric deformation is called \emph{flexible}, otherwise it is called \emph{rigid}.
\emph{Infinitesimal} isometric deformations and infinitesimal rigidity and flexibility are defined in a similar way, by means of infinitesimal rigid motions (that is, velocity vector fields of one-parameter families of rigid motions).

Discrete conjugate nets are a structure-preserving discretization of conjugate nets in the classical differential geometry.
An isometric deformation of a discrete conjugate net is similar to an isometric deformations of a smooth surface preserving the conjugacy of some conjugate net.
The study of the latter deformations was initiated by Peterson and completed by Bianchi and Cosserat, see \cite{Finikoff1937, Finikoff1939} for an overview and references.

% In the classical differential geometry, a net on a smooth surface is its covering by curves such that through every point two (germs of) curves are passing.
% A net on a surface in $\RR^3$ is called conjugate, if at every point of the surface the tangent vectors to the two curves through this point are orthogonal with respect to the second fundamental form.
% It was noted back in the XIX century that the small quadrilaterals cut out by two pairs of curves in a conjugate net are planar of a higher order than small quadrilaterals in an arbitrary net.
% This explains why surfaces made of planar quadrilaterals are called discrete conjugate nets.

% An isometric deformation of a discrete conjugate net can thus be viewed as a discrete analog of an isometric deformation of a smooth surface preserving the conjugacy of some conjugate net.
% A study of deformations of this kind was initiated by Peterson in \cite{Peterson1868} and completed by Bianchi and Cosserat, see \cite{Finikoff1937, Finikoff1939} for an overview and references.

A simple count of degrees of freedom reveals that a generic $3 \times 3$ conjugate net is rigid and infinitesimally rigid.
Graf and Sauer have described in \cite{GrafSauer1931} two classes of flexible discrete conjugate nets of arbitrary size.
Kokotsakis proved in \cite{Kokotsakis1933} a criterium for infinitesimal flexibility of a $3 \times 3$-net and described a new class of flexible $3 \times 3$-nets.
In the last decades a renewed interest in the topic resulted in several works \cite{SBH08, Stachel2010, Naw11, Naw12} attempting to classify flexible discrete conjugate nets.
A complete classification of flexible $3 \times 3$-nets was finally achieved in \cite{Izmestiev2017}.
A detailed analysis of several classes of flexible $m \times n$-nets including algorithmic aspects of their construction was undertaken in \cite{Sharifmoghaddam2021, Izmestiev2024, Kilian2024}, see Figure \ref{fig:ExamplesDCN} for two examples.

\begin{figure}[ht]
\begin{center}
\includegraphics[width=.45\textwidth]{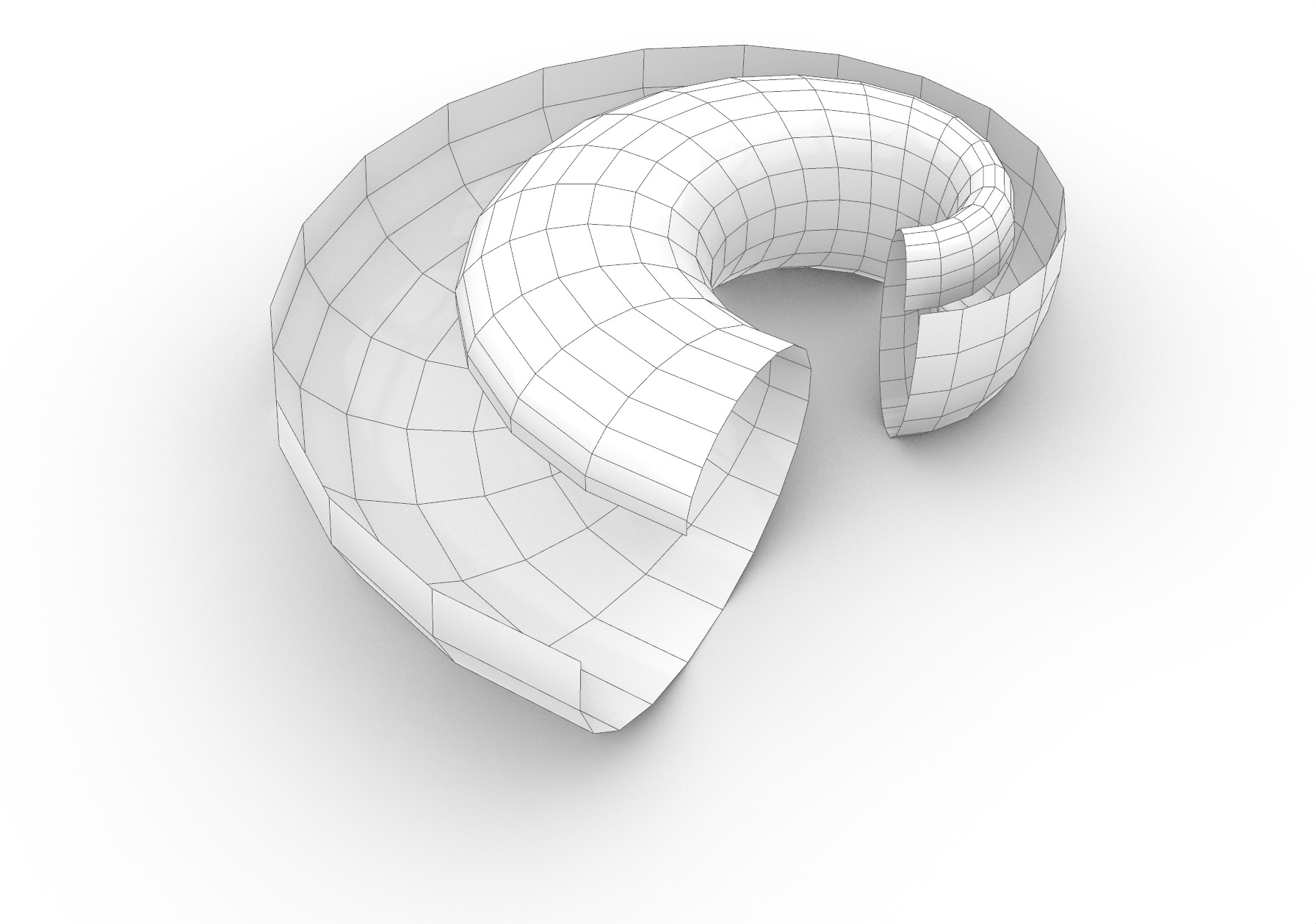} \includegraphics[width=.45\textwidth]{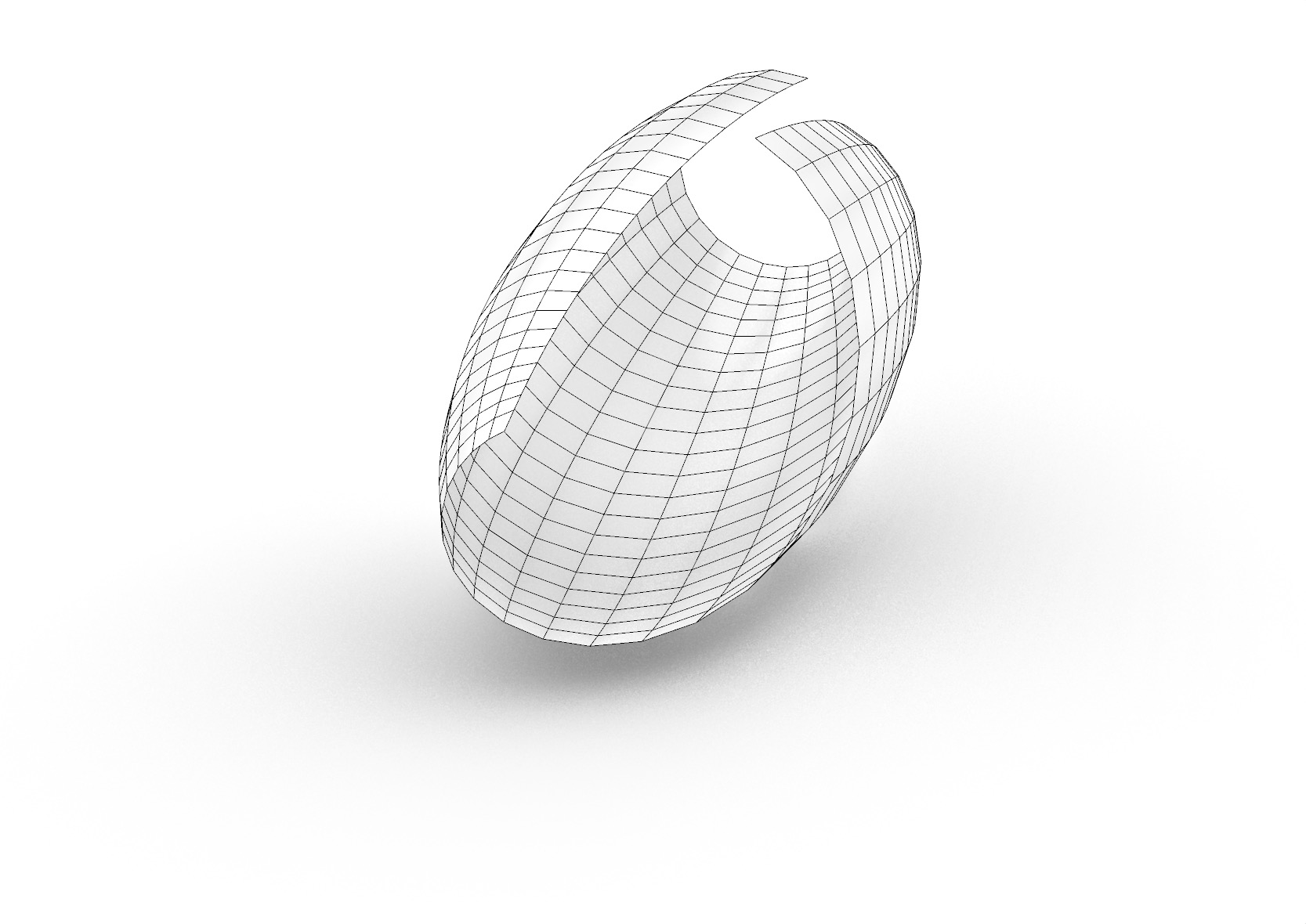}
\end{center}
\caption{Two examples of flexible discrete conjugate nets.}
\label{fig:ExamplesDCN}
\end{figure}

The article \cite{SBH08} by Schief, Bobenko, and Hoffmann discusses connections of flexible discrete conjugate nets to the theory of integrable systems.
Unfortunately, this article contains an inaccurate statement which has been widely cited, namely Theorem 3.2 that reads as follows.

\begin{thm}
\label{thm:SBH}
A non-degenerate discrete conjugate net is isometrically deformable if and only if all of its $3\times 3$-complexes are isometrically deformable.
\end{thm}

A discrete conjugate net is called in \cite{SBH08} degenerate if it has a pair of collinear edges that share a vertex.

Section \ref{sec:Counterexamples} of the present article contains two counterexamples to Theorem \ref{thm:SBH}, namely two rigid $3 \times 4$ conjugate nets which are non-degenerate in the above sense but both of whose $3 \times 3$-subnets are flexible.
In Section \ref{sec:Globalization} a corrected version of Theorem \ref{thm:SBH} is proved: if a discrete conjugate net has no pair of coplanar adjacent faces, then the flexibility of all $3 \times 3$-subnets implies the flexibility of the whole net.

Similar questions can be raised about the infinitesimal flexibility of discrete conjugate nets.
They are answered in Section \ref{sec:IID} in a similar way.
Firstly, in the absence of coplanar adjacent faces the infinitesimal flexibility of all $3 \times 3$-subnets implies the infinitesimal flexibility of the whole net.
Secondly there is an example of an infinitesimally rigid $3 \times 4$-discrete conjugate net with infinitesimally flexible $3 \times 3$-subcomplexes.

The final Section 5 contains various observations related to the problem: comparison of the counterexamples, analogs in planar kinematics, and discussion of non-simply-connected nets.

\section{Globalization theorem}
\label{sec:Globalization}
Throughout this section it is assumed that the quadrilaterals of which a discrete conjugate net is composed are strictly convex (that is with all angles $< 180^\circ$).
Through simple modifications our results can be extended to nets containing both strictly convex and non-convex faces (that is with angles $\ne 180^\circ$).
Self-intersections of a discrete conjugate net are allowed.

\subsection{Configuration space of a $2 \times 2$-complex}
A $2 \times 2$-complex of quadrilaterals has four interior edges incident to the single interior vertex.
A configuration of a $2 \times 2$-complex is uniquely determined by the values of the dihedral angles at these four edges.
The four dihedral angles are, of course, not independent of each other.
The following construction leads to a better understanding of their mutual dependence.

Take a unit sphere centered at the interior vertex of the complex.
Each of the interior edges or its extension meets this sphere in a point, and each face (or its extension) intersects the sphere along an arc of a great circle connecting the corresponding points.
One obtains a spherical quadrilateral whose side lengths and angle values are equal to the face angles, respectively the dihedral angles, at the interior vertex of the $2\times 2$-complex.
Let us call this quadrilateral the \emph{tangent image} of the vertex.
(Another commonly used term for this is spherical image.)
An isometric deformation of the $2 \times 2$-complex thus corresponds to a side length preserving deformation of the spherical quadrilateral which is the tangent image of its interior vertex.

\begin{dfn}
Let an orientation of the ambient space and an orientation of the polyhedral surface be fixed.
For an interior edge of the surface, choose any of the two possible orientations.
Define the \emph{exterior dihedral angle} at this edge as the angle by which the face on the right hand side has to be rotated so that to become coplanar but not overlapping with the face on the left hand side.
\end{dfn}
Observe that the value of the exterior dihedral angle does not depend on the choice of the edge orientation.

Figure \ref{fig:2x2TangentImage} shows a $2 \times 2$-complex together with the tangent image of its interior vertex and introduces the notation for the face angles as well as the exterior dihedral angles.

\begin{figure}[ht]
\begin{center}
\begin{picture}(0,0)%
\includegraphics{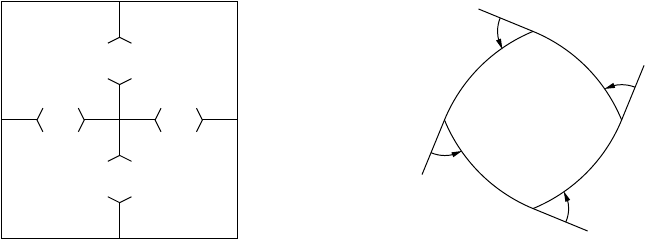}%
\end{picture}%
\setlength{\unitlength}{4144sp}%
\begin{picture}(4920,1824)(-461,-973)
\put(4152,261){\makebox(0,0)[lb]{\smash{\fontsize{9}{10.8}\usefont{T1}{ptm}{m}{n}{\color[rgb]{0,0,0}$\phi_+$}%
}}}
\put(3097,569){\makebox(0,0)[lb]{\smash{\fontsize{9}{10.8}\usefont{T1}{ptm}{m}{n}{\color[rgb]{0,0,0}$\psi_+$}%
}}}
\put(3896,-777){\makebox(0,0)[lb]{\smash{\fontsize{9}{10.8}\usefont{T1}{ptm}{m}{n}{\color[rgb]{0,0,0}$\psi_-$}%
}}}
\put(2964,239){\makebox(0,0)[lb]{\smash{\fontsize{9}{10.8}\usefont{T1}{ptm}{m}{n}{\color[rgb]{0,0,0}$\beta$}%
}}}
\put(3165,-631){\makebox(0,0)[lb]{\smash{\fontsize{9}{10.8}\usefont{T1}{ptm}{m}{n}{\color[rgb]{0,0,0}$\gamma$}%
}}}
\put(3902,440){\makebox(0,0)[lb]{\smash{\fontsize{9}{10.8}\usefont{T1}{ptm}{m}{n}{\color[rgb]{0,0,0}$\alpha$}%
}}}
\put(4072,-507){\makebox(0,0)[lb]{\smash{\fontsize{9}{10.8}\usefont{T1}{ptm}{m}{n}{\color[rgb]{0,0,0}$\delta$}%
}}}
\put(2839,-458){\makebox(0,0)[lb]{\smash{\fontsize{9}{10.8}\usefont{T1}{ptm}{m}{n}{\color[rgb]{0,0,0}$\phi_-$}%
}}}
\put(496,-16){\makebox(0,0)[lb]{\smash{\fontsize{9}{10.8}\usefont{T1}{ptm}{m}{n}{\color[rgb]{0,0,0}$\alpha$}%
}}}
\put(496,-196){\makebox(0,0)[lb]{\smash{\fontsize{9}{10.8}\usefont{T1}{ptm}{m}{n}{\color[rgb]{0,0,0}$\delta$}%
}}}
\put(361,344){\makebox(0,0)[lb]{\smash{\fontsize{9}{10.8}\usefont{T1}{ptm}{m}{n}{\color[rgb]{0,0,0}$\psi_+$}%
}}}
\put(361,-556){\makebox(0,0)[lb]{\smash{\fontsize{9}{10.8}\usefont{T1}{ptm}{m}{n}{\color[rgb]{0,0,0}$\psi_-$}%
}}}
\put(-89,-106){\makebox(0,0)[lb]{\smash{\fontsize{9}{10.8}\usefont{T1}{ptm}{m}{n}{\color[rgb]{0,0,0}$\phi_-$}%
}}}
\put(316,-16){\makebox(0,0)[lb]{\smash{\fontsize{9}{10.8}\usefont{T1}{ptm}{m}{n}{\color[rgb]{0,0,0}$\beta$}%
}}}
\put(316,-196){\makebox(0,0)[lb]{\smash{\fontsize{9}{10.8}\usefont{T1}{ptm}{m}{n}{\color[rgb]{0,0,0}$\gamma$}%
}}}
\put(811,-106){\makebox(0,0)[lb]{\smash{\fontsize{9}{10.8}\usefont{T1}{ptm}{m}{n}{\color[rgb]{0,0,0}$\phi_+$}%
}}}
\end{picture}%
\end{center}
\caption{A $2 \times 2$-complex and the tangent image of its interior vertex.}
\label{fig:2x2TangentImage}
\end{figure}
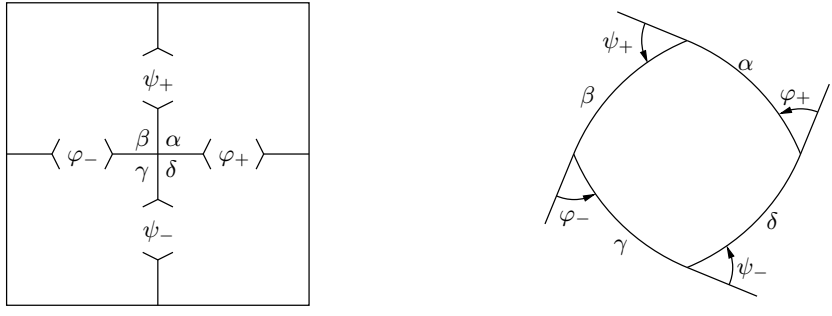

With the exterior dihedral angles $(\phi_+, \psi_+, \phi_-, \psi_-)$ used as coordinates, the configuration space of a spherical quadrilateral becomes a subset of $(\Sph^1)^4$, where $\Sph^1 = \RR/2\pi\ZZ$.
The configuration spaces of Euclidean and spherical quadrilaterals are well-studied.
In particular, they become algebraic subsets of $(\RR\mathrm{P}^1)^4$ after the tangent half-angle substitution.
In the spherical case, the equations relating each pair of dihedral angles were derived by Bricard \cite{Bricard1897}.
Note however that the configuration space is not always identical with the solution set of Bricard's system of equations, see \cite[Proposition 2.6]{Izmestiev2023} which deals with the Euclidean case.
A complete description of the configuration spaces of spherical quadrilaterals was given in \cite{Izmestiev2017}.

In this section, we are interested in the analytic rather than the algebraic properties of the configuration space, and only its part where all of the dihedral angles are different from $0$ and $\pi$ is of relevance.

\begin{thm}
\label{thm:NoFlatSubmanifold}
In a neighborhood of every point where all angles $\phi_+, \psi_+, \phi_-, \psi_-$ are different from $0$ and $\pi$ the configuration space of a spherical quadrilateral is a one-dimensional smooth submanifold of $(\Sph^1)^4$, which can be locally parametrized by any of the angles $\phi_+, \psi_+, \phi_-, \psi_-$.
\end{thm}

\begin{proof}
It suffices to show that near every point where all angles $\phi_+, \psi_+, \phi_-, \psi_-$ are different from $0$ and $\pi$ any three of these angles are smooth functions of the fourth angle, and the derivatives of these functions do not vanish.

Let us study the dependence of $\phi_-$ and $\psi_+$ on $\phi_+$
The other dependencies follow from these two by symmetry.
Since $\phi_+, \phi_- \notin \{0,\pi\}$, the quadrilateral can be cut into two non-degenerate triangles by the diagonal from the $\alpha\beta$-vertex to the $\gamma\delta$-vertex.
Denote the length of this diagonal by $\lambda$, and the interior angles adjacent to this diagonal by $\eta_\pm, \theta_\pm$ as shown in Figure \ref{fig:CutTriangles}.

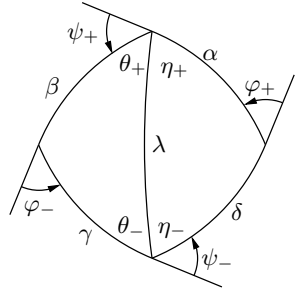
\begin{figure}[ht]
\begin{center}
\begin{picture}(0,0)%
\includegraphics{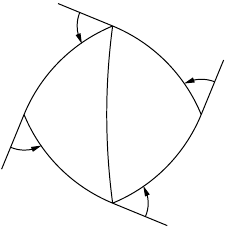}%
\end{picture}%
\setlength{\unitlength}{4144sp}%
\begin{picture}(1716,1716)(2743,-919)
\put(4152,261){\makebox(0,0)[lb]{\smash{\fontsize{9}{10.8}\usefont{T1}{ptm}{m}{n}{\color[rgb]{0,0,0}$\phi_+$}%
}}}
\put(3097,569){\makebox(0,0)[lb]{\smash{\fontsize{9}{10.8}\usefont{T1}{ptm}{m}{n}{\color[rgb]{0,0,0}$\psi_+$}%
}}}
\put(3896,-777){\makebox(0,0)[lb]{\smash{\fontsize{9}{10.8}\usefont{T1}{ptm}{m}{n}{\color[rgb]{0,0,0}$\psi_-$}%
}}}
\put(2964,239){\makebox(0,0)[lb]{\smash{\fontsize{9}{10.8}\usefont{T1}{ptm}{m}{n}{\color[rgb]{0,0,0}$\beta$}%
}}}
\put(3165,-631){\makebox(0,0)[lb]{\smash{\fontsize{9}{10.8}\usefont{T1}{ptm}{m}{n}{\color[rgb]{0,0,0}$\gamma$}%
}}}
\put(3902,440){\makebox(0,0)[lb]{\smash{\fontsize{9}{10.8}\usefont{T1}{ptm}{m}{n}{\color[rgb]{0,0,0}$\alpha$}%
}}}
\put(4072,-507){\makebox(0,0)[lb]{\smash{\fontsize{9}{10.8}\usefont{T1}{ptm}{m}{n}{\color[rgb]{0,0,0}$\delta$}%
}}}
\put(2839,-458){\makebox(0,0)[lb]{\smash{\fontsize{9}{10.8}\usefont{T1}{ptm}{m}{n}{\color[rgb]{0,0,0}$\phi_-$}%
}}}
\put(3601,-106){\makebox(0,0)[lb]{\smash{\fontsize{9}{10.8}\usefont{T1}{ptm}{m}{n}{\color[rgb]{0,0,0}$\lambda$}%
}}}
\put(3406,380){\makebox(0,0)[lb]{\smash{\fontsize{9}{10.8}\usefont{T1}{ptm}{m}{n}{\color[rgb]{0,0,0}$\theta_+$}%
}}}
\put(3395,-588){\makebox(0,0)[lb]{\smash{\fontsize{9}{10.8}\usefont{T1}{ptm}{m}{n}{\color[rgb]{0,0,0}$\theta_-$}%
}}}
\put(3652,368){\makebox(0,0)[lb]{\smash{\fontsize{9}{10.8}\usefont{T1}{ptm}{m}{n}{\color[rgb]{0,0,0}$\eta_+$}%
}}}
\put(3636,-585){\makebox(0,0)[lb]{\smash{\fontsize{9}{10.8}\usefont{T1}{ptm}{m}{n}{\color[rgb]{0,0,0}$\eta_-$}%
}}}
\end{picture}%
\end{center}
\caption{A spherical quadrilateral cut into two triangles.}
\label{fig:CutTriangles}
\end{figure}

In the $\alpha\delta\lambda$-triangle the side lengths $\alpha$ and $\delta$ are fixed.
The configuration space of this triangle in the neighborhood of the given point (where $\phi_+ \ne 0, \pi$) can be parametrized by any of two variables $\phi_+$ and $\lambda$.
Similarly for the $\beta\gamma\lambda$-triangle.
By Lemma \ref{lem:SphTrig} below one has
\[
\frac{d\phi_+}{d\lambda} = - \frac{1}{\sin\alpha\sin\eta_+}, \quad
\frac{d\phi_-}{d\lambda} = -\frac{1}{\sin\beta\sin\theta_+},
\]
which implies
\[
\frac{d\phi_-}{d\phi_+} = \frac{\sin\alpha\sin\eta_+}{\sin\beta\sin\theta_+} \ne 0.
\]
By Lemma \ref{lem:SphTrig} one also has
\[
\frac{d\psi_+}{d\lambda} = -\frac{d\eta_+}{d\lambda} - \frac{d\theta_+}{d\lambda}
= -\frac{\cot\eta_- + \cot\theta_-}{\sin\lambda}
= -\frac{\sin\psi_-}{\sin\lambda \sin\eta_- \sin\theta_-}
\]
which implies
\[
\frac{d\psi_+}{d\phi_+} = \frac{\sin\psi_- \sin\alpha \sin\eta_+}{\sin\lambda \sin\eta_- \sin\theta_-}
= \frac{\sin\psi_- \sin\eta_+}{\sin\phi_+ \sin\theta_-} \ne 0.
\]
\end{proof}

\begin{lem}
\label{lem:SphTrig}
Let in a spherical triangle with side lengths $a, b, c$ and the values of respective opposite angles $\alpha, \beta, \gamma$ the lengths $a$ and $b$ be fixed.
Then the derivatives of the angle values with respect to $c$ are
\[
\frac{d\alpha}{dc} = -\frac{\cot\beta}{\sin c}, \quad
\frac{d\beta}{dc} = -\frac{\cot\alpha}{\sin c}, \quad
\frac{d\gamma}{dc} = \frac{1}{\sin a \sin\beta} = \frac{1}{\sin b \sin\alpha}.
\]
\end{lem}
\begin{proof}
Let us prove the third equation first.
It is done by differentiating the spherical cosine law:
\begin{multline*}
\cos c = \cos a \cos b + \sin a \sin b \cos\gamma \Rightarrow
-\sin c = - \frac{d\gamma}{dc}\sin a \sin b \sin\gamma\\
\Rightarrow \frac{d\gamma}{dc} = \frac{\sin c}{\sin a \sin b \sin \gamma} = \frac{1}{\sin a \sin\beta} = \frac{1}{\sin b \sin\alpha},
\end{multline*}
where in the last two equations the spherical sine law was used.

For the first equation differentiate the spherical sine law and use the already proved formula for $\frac{d\gamma}{dc}$:
\begin{multline*}
\sin\alpha = \sin a \frac{\sin\gamma}{\sin c} \Rightarrow
\frac{d\alpha}{dc} = \frac{\sin a}{\cos\alpha} \frac{\frac{\cos\gamma \sin c}{\sin a \sin\beta} - \sin\gamma \cos c}{\sin^2 c} = \frac{\cos\gamma \sin c - \sin a \sin\beta \sin\gamma \cos c}{\cos\alpha \sin\beta \sin^2 c}\\
= \frac{\cos\gamma \sin c - \sin c \sin\beta \sin\alpha \cos c}{\cos\alpha \sin\beta \sin^2 c} = \frac{\cos\gamma - \sin\alpha \sin\beta \cos c}{\cos\alpha \sin\beta \sin c} = \frac{-\cos\alpha \cos\beta}{\cos\alpha \sin\beta \sin c} = - \frac{\cot\beta}{\sin c},
\end{multline*}
where the spherical sine law and the dual spherical cosine law $\cos\gamma = -\cos\alpha\cos\beta + \sin\alpha\sin\beta cos c$ were used.

The second equation is proved similarly to the first one.
\end{proof}

\subsection{Globalization in the absence of coplanar adjacent faces}

\begin{thm}
\label{thm:GlobFlex}
If an $m \times n$ discrete conjugate net has no pair of coplanar adjacent faces and all of its $3 \times 3$-subnets are flexible, then the net is flexible as a whole.
\end{thm}
\begin{proof}
For every interior edge $e$ of the given $m \times n$-net denote by $\phi_e^0$ the exterior dihedral angle at $e$.
It suffices to show that there are functions $\phi_e(t) \colon (-\varepsilon, \varepsilon) \to \RR$ for some $\varepsilon > 0$ that satisfy $\phi_e(0) = \phi_e^0$ and are compatible around every interior vertex of the net, that is the quadruple of functions on the edges adjacent to a vertex parametrizes a subset of the configuration space of the corresponding $2\times 2$-subcomplex.

Since $\phi_e(0) \notin \{0, \pi\}$, Theorem \ref{thm:NoFlatSubmanifold} says that the angles at the edges $e_1, e_2, e_3, e_4$ adjacent to the same vertex can be parametrized as $\phi_1(t), \phi_2(t), \phi_3(t), \phi_4(t)$ so that each of these functions is a diffeomorphism from $(-\varepsilon, \varepsilon)$ to some interval of the real line.
This also means that there is a unique diffeomorphism $f_{21}$ such that $\phi_2 = f_{21} \circ \phi_1$.
In order to define functions $\phi_e(t)$ for all edges $e$ simultaneously, start with an arbitrary edge $e_1$ and choose an arbitrary edge path $e_1, e_2, \ldots, e_n = e$.
Define $\phi_1(t)$ around $t=0$ arbitrarily, for example as $\phi_1(t) = \phi_1^0 + t$, and put
\[
\phi_n = f_{n,n-1} \circ \cdots \circ f_{21} \circ \phi_1.
\]
It has to be checked that this is well-defined, that is independent of the choice of an edge path connecting $e_1$ with $e$.
This follows by a standard argument from the local compatibility:
\[
f_{14} \circ f_{43} \circ f_{32} \circ f_{21} = \mathrm{id},
\]
where $e_1, e_2, e_3, e_4$ are the consequtive edges of any square of the $m \times n$-net.
The local compatibility, in turn, is satisfied because of the flexibility of every $3\times 3$-subnet.
\end{proof}
%
% Since no pair of adjacent faces is coplanar, none of the dihedral angles vanish.
% Thus by Theorem \ref{thm:NoFlatSubmanifold} the configuration space of any $2\times 2$-subcomplex can be locally parametrized by any of the dihedral angles involved.
% In particular, for any two edges $e_1$ and $e_2$ sharing a vertex there is a smooth function $f_{21}$ such that in the configuration space of the $2 \times 2$-subcomplex surrounding that vertex the dihedral angles at $e_1$ and $e_2$
%
% (need to make this more precise)
%
% satisfy the equation
% \begin{equation}
% \label{eqn:Transport}
%$\phi_2 = f_{21}(\phi_1)$.
% \end{equation}

%
%
% We will show that it is possible to assign to every interior edge of the net a smooth function so that changing all dihedral angles according to these functions results in an isometric deformation of the net.
% For this it suffices to ensure that the functions are compatible around each vertex,

\section{Counterexamples to the globalization theorem of Bobenko--Hoffmann--Schief}
\label{sec:Counterexamples}
\subsection{First counterexample}\label{sec:ex1}
Our first example of a rigid $3 \times 4$-net with flexible $3 \times 3$-subnets is shown in Figure \ref{fig:example13d}.
Figure \ref{fig:example1schem} introduces the notation for planar and exterior dihedral angles of this polyhedron.
The angle assignment is symmetric with respect to the horizontal dashed line; this results in the symmetry of the polyhedron with respect to a plane.
Our particular choice of angles is
\[
\alpha = \frac{\pi}4, \quad \beta = \frac{2\pi}{3}, \quad \gamma = \frac{\pi}{3}, \quad \delta = \frac{\pi}4,
\]
but any other choice with $2\alpha + \delta < \pi$ will work as well.

\begin{figure}[ht]
    \centering
    \includegraphics[width=0.58\textwidth]{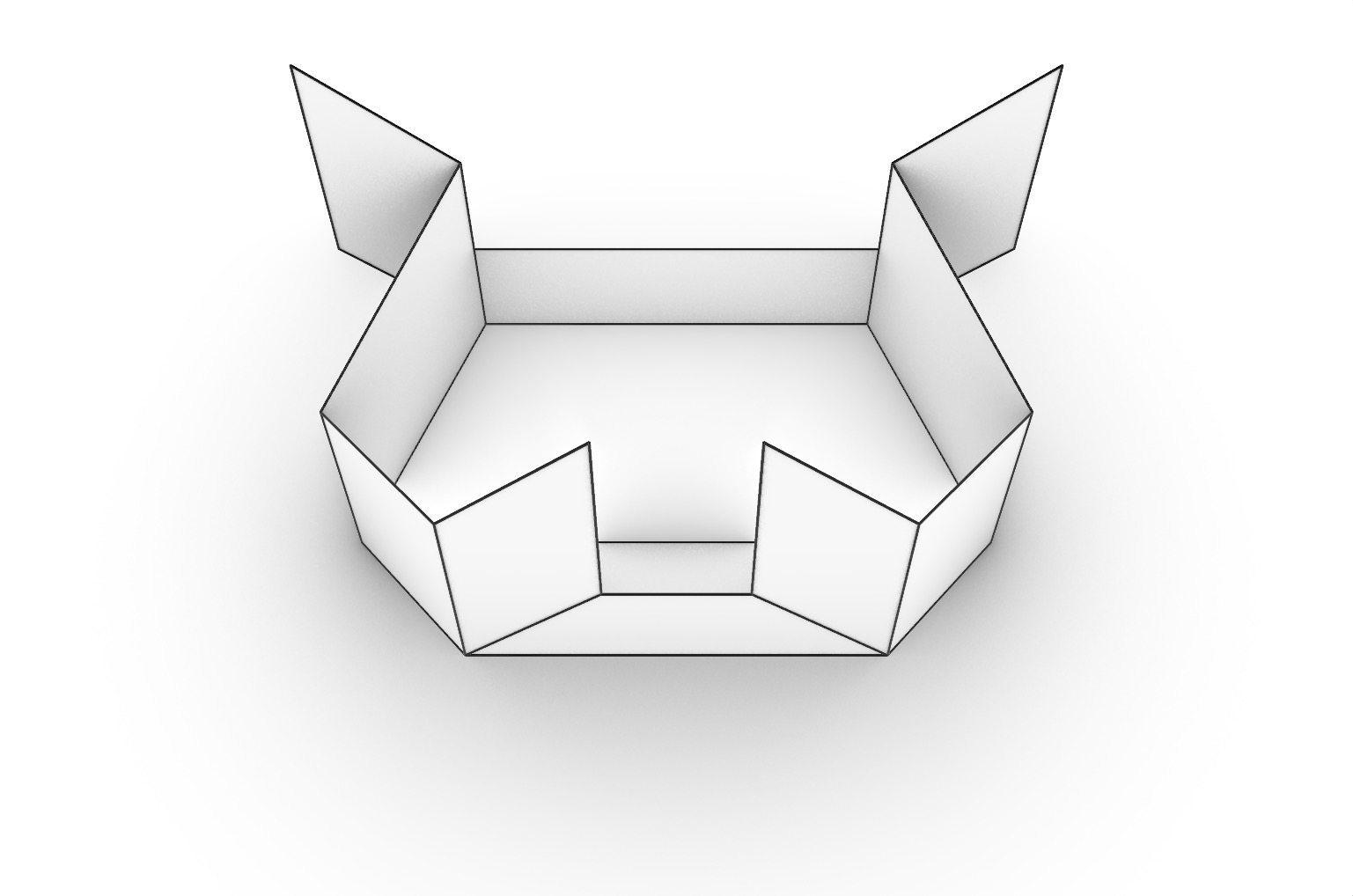}
        \includegraphics[width=0.4\textwidth]{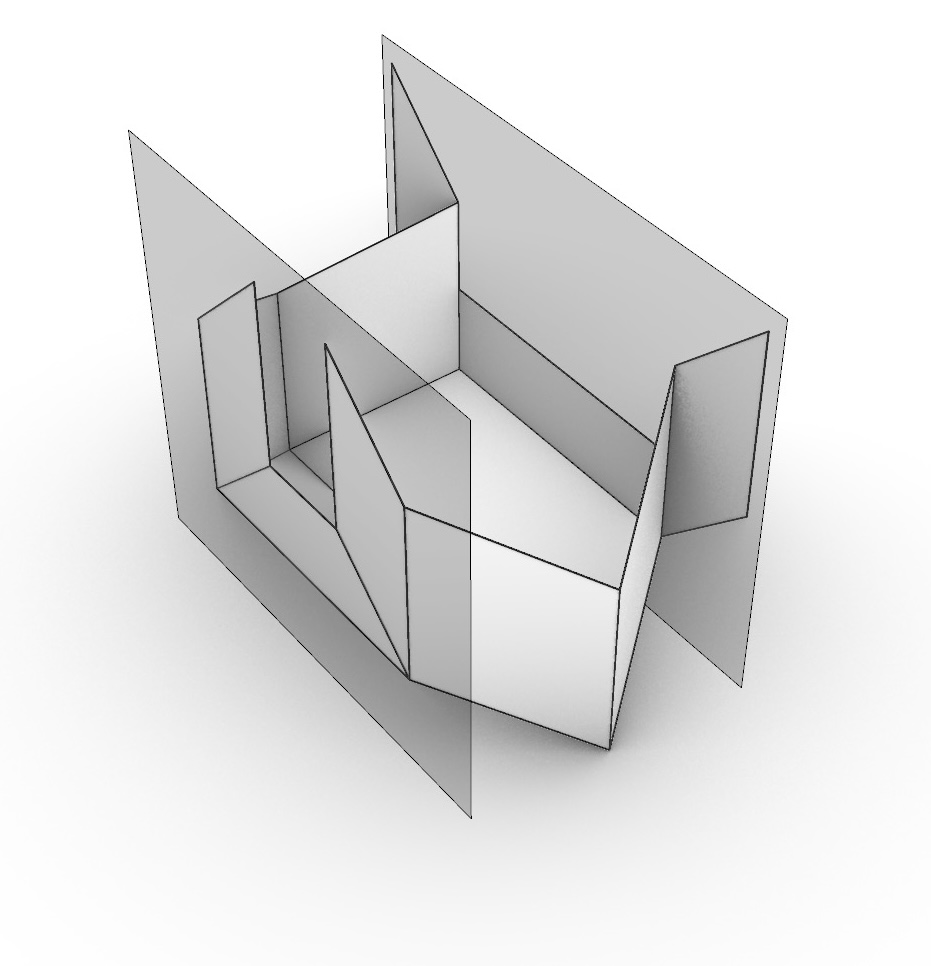}
    \caption{Two views of the first counterexample. In the right figure each of the grey planes contains three faces.}
    \label{fig:example13d}
\end{figure}

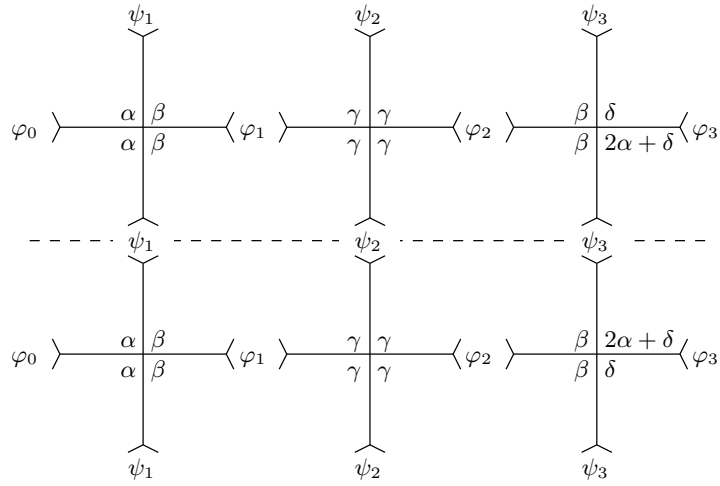
\begin{figure}[ht]
\begin{center}
\begin{picture}(0,0)%
\includegraphics{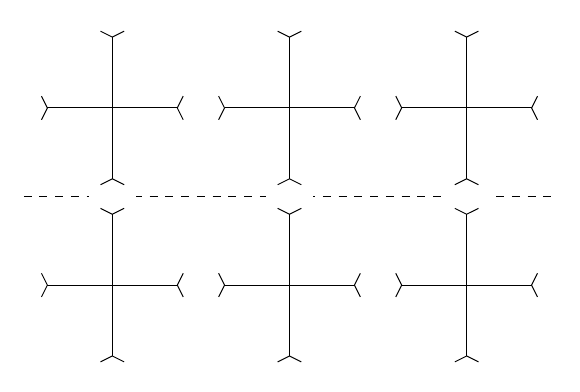}%
\end{picture}%
\setlength{\unitlength}{4144sp}%
\begin{picture}(4412,2972)(-180,-4472)
\put(721,-3796){\makebox(0,0)[lb]{\smash{\fontsize{9}{10.8}\usefont{T1}{ptm}{m}{n}{\color[rgb]{0,0,0}$\beta$}%
}}}
\put(721,-3616){\makebox(0,0)[lb]{\smash{\fontsize{9}{10.8}\usefont{T1}{ptm}{m}{n}{\color[rgb]{0,0,0}$\beta$}%
}}}
\put(721,-2266){\makebox(0,0)[lb]{\smash{\fontsize{9}{10.8}\usefont{T1}{ptm}{m}{n}{\color[rgb]{0,0,0}$\beta$}%
}}}
\put(721,-2446){\makebox(0,0)[lb]{\smash{\fontsize{9}{10.8}\usefont{T1}{ptm}{m}{n}{\color[rgb]{0,0,0}$\beta$}%
}}}
\put(2071,-3616){\makebox(0,0)[lb]{\smash{\fontsize{9}{10.8}\usefont{T1}{ptm}{m}{n}{\color[rgb]{0,0,0}$\gamma$}%
}}}
\put(2071,-3796){\makebox(0,0)[lb]{\smash{\fontsize{9}{10.8}\usefont{T1}{ptm}{m}{n}{\color[rgb]{0,0,0}$\gamma$}%
}}}
\put(2071,-2266){\makebox(0,0)[lb]{\smash{\fontsize{9}{10.8}\usefont{T1}{ptm}{m}{n}{\color[rgb]{0,0,0}$\gamma$}%
}}}
\put(2071,-2446){\makebox(0,0)[lb]{\smash{\fontsize{9}{10.8}\usefont{T1}{ptm}{m}{n}{\color[rgb]{0,0,0}$\gamma$}%
}}}
\put(3421,-3796){\makebox(0,0)[lb]{\smash{\fontsize{9}{10.8}\usefont{T1}{ptm}{m}{n}{\color[rgb]{0,0,0}$\delta$}%
}}}
\put(3421,-3616){\makebox(0,0)[lb]{\smash{\fontsize{9}{10.8}\usefont{T1}{ptm}{m}{n}{\color[rgb]{0,0,0}$2\alpha+\delta$}%
}}}
\put(3421,-2266){\makebox(0,0)[lb]{\smash{\fontsize{9}{10.8}\usefont{T1}{ptm}{m}{n}{\color[rgb]{0,0,0}$\delta$}%
}}}
\put(3421,-2446){\makebox(0,0)[lb]{\smash{\fontsize{9}{10.8}\usefont{T1}{ptm}{m}{n}{\color[rgb]{0,0,0}$2\alpha+\delta$}%
}}}
\put(541,-2266){\makebox(0,0)[lb]{\smash{\fontsize{9}{10.8}\usefont{T1}{ptm}{m}{n}{\color[rgb]{0,0,0}$\alpha$}%
}}}
\put(541,-2446){\makebox(0,0)[lb]{\smash{\fontsize{9}{10.8}\usefont{T1}{ptm}{m}{n}{\color[rgb]{0,0,0}$\alpha$}%
}}}
\put(541,-3616){\makebox(0,0)[lb]{\smash{\fontsize{9}{10.8}\usefont{T1}{ptm}{m}{n}{\color[rgb]{0,0,0}$\alpha$}%
}}}
\put(541,-3796){\makebox(0,0)[lb]{\smash{\fontsize{9}{10.8}\usefont{T1}{ptm}{m}{n}{\color[rgb]{0,0,0}$\alpha$}%
}}}
\put(1891,-2266){\makebox(0,0)[lb]{\smash{\fontsize{9}{10.8}\usefont{T1}{ptm}{m}{n}{\color[rgb]{0,0,0}$\gamma$}%
}}}
\put(1891,-2446){\makebox(0,0)[lb]{\smash{\fontsize{9}{10.8}\usefont{T1}{ptm}{m}{n}{\color[rgb]{0,0,0}$\gamma$}%
}}}
\put(1891,-3616){\makebox(0,0)[lb]{\smash{\fontsize{9}{10.8}\usefont{T1}{ptm}{m}{n}{\color[rgb]{0,0,0}$\gamma$}%
}}}
\put(1891,-3796){\makebox(0,0)[lb]{\smash{\fontsize{9}{10.8}\usefont{T1}{ptm}{m}{n}{\color[rgb]{0,0,0}$\gamma$}%
}}}
\put(3241,-2266){\makebox(0,0)[lb]{\smash{\fontsize{9}{10.8}\usefont{T1}{ptm}{m}{n}{\color[rgb]{0,0,0}$\beta$}%
}}}
\put(3241,-2446){\makebox(0,0)[lb]{\smash{\fontsize{9}{10.8}\usefont{T1}{ptm}{m}{n}{\color[rgb]{0,0,0}$\beta$}%
}}}
\put(3241,-3616){\makebox(0,0)[lb]{\smash{\fontsize{9}{10.8}\usefont{T1}{ptm}{m}{n}{\color[rgb]{0,0,0}$\beta$}%
}}}
\put(3241,-3796){\makebox(0,0)[lb]{\smash{\fontsize{9}{10.8}\usefont{T1}{ptm}{m}{n}{\color[rgb]{0,0,0}$\beta$}%
}}}
\put(-110,-2356){\makebox(0,0)[lb]{\smash{\fontsize{9}{10.8}\usefont{T1}{ptm}{m}{n}{\color[rgb]{0,0,0}$\phi_0$}%
}}}
\put(1249,-3706){\makebox(0,0)[lb]{\smash{\fontsize{9}{10.8}\usefont{T1}{ptm}{m}{n}{\color[rgb]{0,0,0}$\phi_1$}%
}}}
\put(1249,-2356){\makebox(0,0)[lb]{\smash{\fontsize{9}{10.8}\usefont{T1}{ptm}{m}{n}{\color[rgb]{0,0,0}$\phi_1$}%
}}}
\put(2593,-2356){\makebox(0,0)[lb]{\smash{\fontsize{9}{10.8}\usefont{T1}{ptm}{m}{n}{\color[rgb]{0,0,0}$\phi_2$}%
}}}
\put(2593,-3706){\makebox(0,0)[lb]{\smash{\fontsize{9}{10.8}\usefont{T1}{ptm}{m}{n}{\color[rgb]{0,0,0}$\phi_2$}%
}}}
\put(3943,-3706){\makebox(0,0)[lb]{\smash{\fontsize{9}{10.8}\usefont{T1}{ptm}{m}{n}{\color[rgb]{0,0,0}$\phi_3$}%
}}}
\put(3943,-2356){\makebox(0,0)[lb]{\smash{\fontsize{9}{10.8}\usefont{T1}{ptm}{m}{n}{\color[rgb]{0,0,0}$\phi_3$}%
}}}
\put(-110,-3706){\makebox(0,0)[lb]{\smash{\fontsize{9}{10.8}\usefont{T1}{ptm}{m}{n}{\color[rgb]{0,0,0}$\phi_0$}%
}}}
\put(586,-4381){\makebox(0,0)[lb]{\smash{\fontsize{9}{10.8}\usefont{T1}{ptm}{m}{n}{\color[rgb]{0,0,0}$\psi_1$}%
}}}
\put(1936,-4381){\makebox(0,0)[lb]{\smash{\fontsize{9}{10.8}\usefont{T1}{ptm}{m}{n}{\color[rgb]{0,0,0}$\psi_2$}%
}}}
\put(3286,-4381){\makebox(0,0)[lb]{\smash{\fontsize{9}{10.8}\usefont{T1}{ptm}{m}{n}{\color[rgb]{0,0,0}$\psi_3$}%
}}}
\put(3286,-1681){\makebox(0,0)[lb]{\smash{\fontsize{9}{10.8}\usefont{T1}{ptm}{m}{n}{\color[rgb]{0,0,0}$\psi_3$}%
}}}
\put(1936,-1681){\makebox(0,0)[lb]{\smash{\fontsize{9}{10.8}\usefont{T1}{ptm}{m}{n}{\color[rgb]{0,0,0}$\psi_2$}%
}}}
\put(586,-1681){\makebox(0,0)[lb]{\smash{\fontsize{9}{10.8}\usefont{T1}{ptm}{m}{n}{\color[rgb]{0,0,0}$\psi_1$}%
}}}
\put(1936,-3031){\makebox(0,0)[lb]{\smash{\fontsize{9}{10.8}\usefont{T1}{ptm}{m}{n}{\color[rgb]{0,0,0}$\psi_2$}%
}}}
\put(3286,-3031){\makebox(0,0)[lb]{\smash{\fontsize{9}{10.8}\usefont{T1}{ptm}{m}{n}{\color[rgb]{0,0,0}$\psi_3$}%
}}}
\put(586,-3031){\makebox(0,0)[lb]{\smash{\fontsize{9}{10.8}\usefont{T1}{ptm}{m}{n}{\color[rgb]{0,0,0}$\psi_1$}%
}}}
\end{picture}%
\end{center}
\caption{Angles at the interior vertices and the interior edges in the polyhedron from Figure \ref{fig:example13d}.}
\label{fig:example1schem}
\end{figure}

\begin{thm}
A $3 \times 4$-net with planar angles as shown in Figure \ref{fig:example1schem} has a realization in $\RR^3$ with dihedral angles
\[
\phi_0 = 0, \quad \phi_1 = \phi_2, \quad \phi_3 = \pi
\]
as shown in Figure \ref{fig:example13d}, and this realization is not flexible.
At the same time, both $3 \times 3$-subnets of this realization are flexible.
\end{thm}
\begin{proof}
For symmetry reasons it suffices to consider the $2\times 4$-subnet below the dashed line.
Figure \ref{fig:2coupling} shows the tangent images of the three interior vertices, coupled at the dihedral angles they are sharing in a scissors-like way.
Any adjacent pair of these three spherical quadrilaterals is flexible while all three together form a rigid framework.
\end{proof}

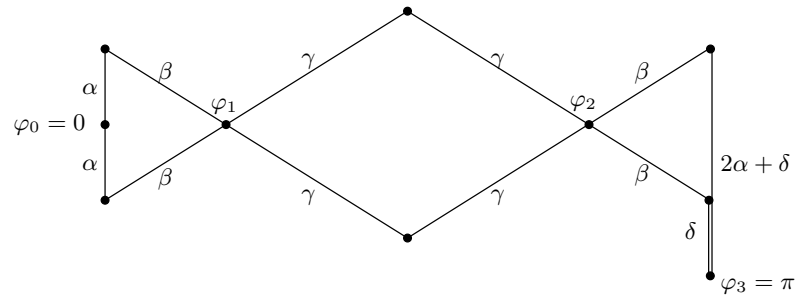
\begin{figure}[ht]
\begin{center}
\begin{picture}(0,0)%
\includegraphics{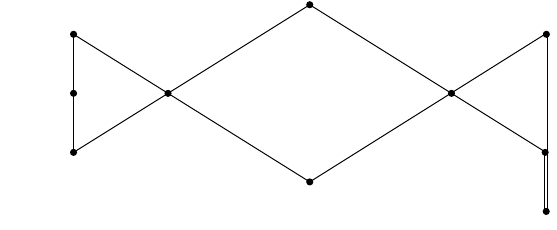}%
\end{picture}%
\setlength{\unitlength}{4144sp}%
\begin{picture}(4238,1741)(-559,-196)
\put(-134,569){\makebox(0,0)[lb]{\smash{\fontsize{9}{10.8}\usefont{T1}{ptm}{m}{n}{\color[rgb]{0,0,0}$\alpha$}%
}}}
\put(-134,1019){\makebox(0,0)[lb]{\smash{\fontsize{9}{10.8}\usefont{T1}{ptm}{m}{n}{\color[rgb]{0,0,0}$\alpha$}%
}}}
\put(316,479){\makebox(0,0)[lb]{\smash{\fontsize{9}{10.8}\usefont{T1}{ptm}{m}{n}{\color[rgb]{0,0,0}$\beta$}%
}}}
\put(316,1109){\makebox(0,0)[lb]{\smash{\fontsize{9}{10.8}\usefont{T1}{ptm}{m}{n}{\color[rgb]{0,0,0}$\beta$}%
}}}
\put(1171,1199){\makebox(0,0)[lb]{\smash{\fontsize{9}{10.8}\usefont{T1}{ptm}{m}{n}{\color[rgb]{0,0,0}$\gamma$}%
}}}
\put(1171,389){\makebox(0,0)[lb]{\smash{\fontsize{9}{10.8}\usefont{T1}{ptm}{m}{n}{\color[rgb]{0,0,0}$\gamma$}%
}}}
\put(2296,1199){\makebox(0,0)[lb]{\smash{\fontsize{9}{10.8}\usefont{T1}{ptm}{m}{n}{\color[rgb]{0,0,0}$\gamma$}%
}}}
\put(2296,389){\makebox(0,0)[lb]{\smash{\fontsize{9}{10.8}\usefont{T1}{ptm}{m}{n}{\color[rgb]{0,0,0}$\gamma$}%
}}}
\put(3151,1109){\makebox(0,0)[lb]{\smash{\fontsize{9}{10.8}\usefont{T1}{ptm}{m}{n}{\color[rgb]{0,0,0}$\beta$}%
}}}
\put(3451,157){\makebox(0,0)[lb]{\smash{\fontsize{9}{10.8}\usefont{T1}{ptm}{m}{n}{\color[rgb]{0,0,0}$\delta$}%
}}}
\put(3664,562){\makebox(0,0)[lb]{\smash{\fontsize{9}{10.8}\usefont{T1}{ptm}{m}{n}{\color[rgb]{0,0,0}$2\alpha+\delta$}%
}}}
\put(2767,950){\makebox(0,0)[lb]{\smash{\fontsize{9}{10.8}\usefont{T1}{ptm}{m}{n}{\color[rgb]{0,0,0}$\phi_2$}%
}}}
\put(3664,-139){\makebox(0,0)[lb]{\smash{\fontsize{9}{10.8}\usefont{T1}{ptm}{m}{n}{\color[rgb]{0,0,0}$\phi_3=\pi$}%
}}}
\put(3151,504){\makebox(0,0)[lb]{\smash{\fontsize{9}{10.8}\usefont{T1}{ptm}{m}{n}{\color[rgb]{0,0,0}$\beta$}%
}}}
\put(630,932){\makebox(0,0)[lb]{\smash{\fontsize{9}{10.8}\usefont{T1}{ptm}{m}{n}{\color[rgb]{0,0,0}$\phi_1$}%
}}}
\put(-544,808){\makebox(0,0)[lb]{\smash{\fontsize{9}{10.8}\usefont{T1}{ptm}{m}{n}{\color[rgb]{0,0,0}$\phi_0 = 0$}%
}}}
\end{picture}%
\end{center}
\caption{Tangent images of three interior vertices below the dashed line in Figure \ref{fig:example1schem}.}
\label{fig:2coupling}
\end{figure}

For a future reference let us write down the equations connecting the angles $\phi_i$ in Figure \ref{fig:2coupling}.
Clearly, the central quadrilateral enforces $\phi_1 = \phi_2$ in the neighborhood of the given configuration.
Subdividing each of the other two quadrilaterals by a diagonal and applying the spherical law of cosines one obtains
\begin{gather*}
\sin^2\alpha \cos\phi_0 - \sin^2\beta \cos\phi_1 = \sin^2\beta - \sin^2\alpha,\\
\sin^2\beta \cos\phi_2 - \sin\delta \sin(2\alpha + \delta) \cos\phi_3 = \cos^2\beta - \cos\delta \cos(2\alpha + \delta),
\end{gather*}
which for our particular values of $\alpha, \beta, \delta$ becomes
\[
2 \cos\phi_0 - 3 \cos\phi_1 = 1, \qquad
3 \cos\phi_2 - 2 \cos\phi_3 = 3.
\]
Thus one has the following system of equations on the variable angles $\phi_0, \phi_1, \phi_2, \phi_3$:
\begin{equation}
\label{eqn:CosineSystem}
\begin{cases}
2 \cos\phi_0 - 3 \cos\phi_1 = 1\\
\phi_1 = \phi_2\\
3 \cos\phi_2 - 2 \cos\phi_3 = 3
\end{cases}
\end{equation}
In terms of the cosines of the angles this system is linear and has a one-parameter solution set.
However only one solution has all cosines within the range $[-1, 1]$, namely
\[
\cos\phi_0 = 1, \quad \cos\phi_1 = \cos\phi_2 = \frac13, \quad \cos\phi_3 = -1.
\]
At the same time, removing from \eqref{eqn:CosineSystem} the first or the last equation expands the solution set to a half-line.
This corresponds to the inflexibility of the linkage in Figure \ref{fig:2coupling} and the flexibility of the corresponding sublinkages.

\subsection{Second counterexample}\label{sec:ex2}
Our second example is depicted in Figure \ref{fig:example-2}.
In the middle of the top row is a planar $3 \times 4$-net which will be shown to be rigid, while both of its $3 \times 3$-subnets (top row left and right) are flexible, with their deformed states shown in the bottom row.

\begin{figure}[ht]
    \centering
    \includegraphics[width=0.32\textwidth]{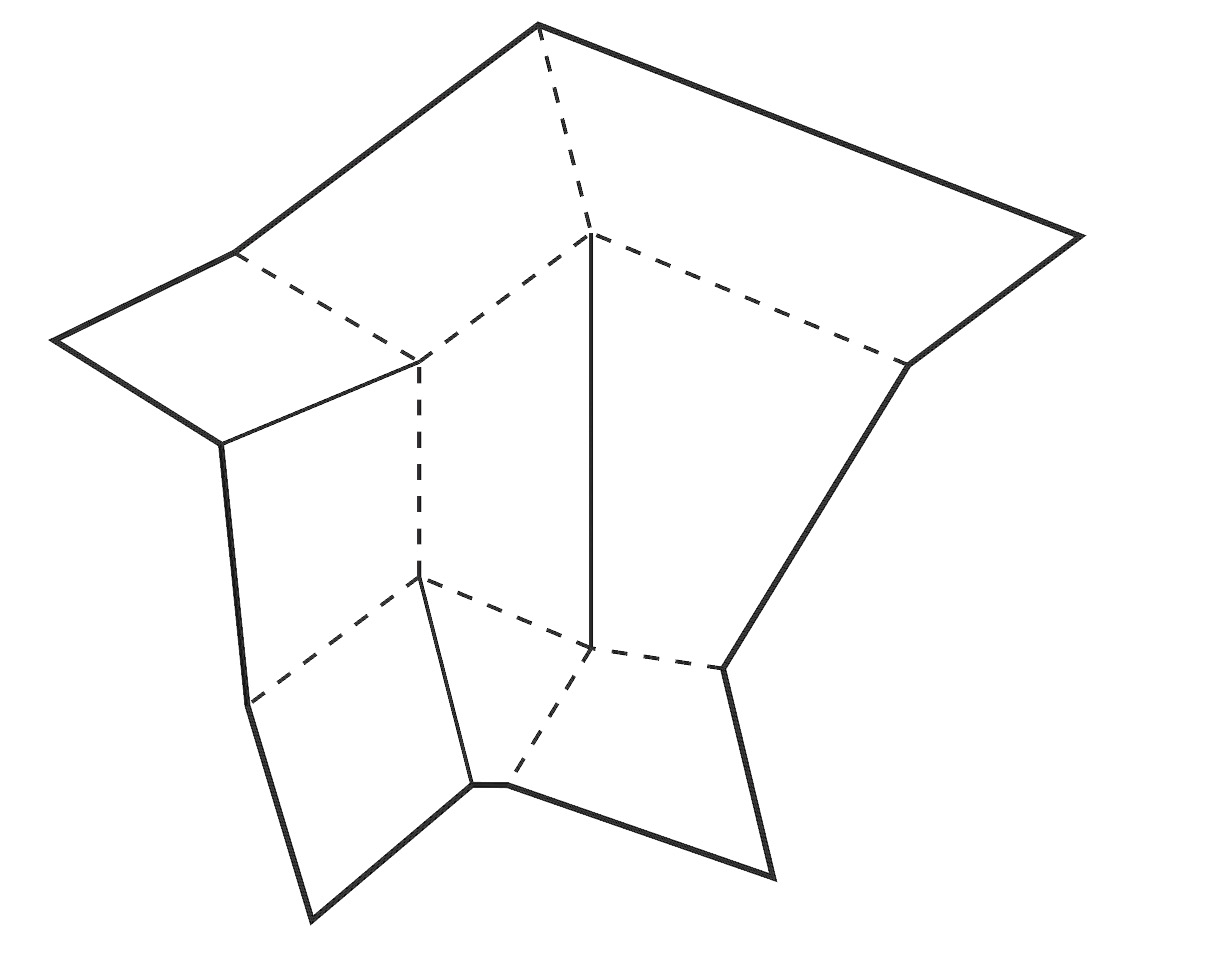}
        \includegraphics[width=0.32\textwidth]{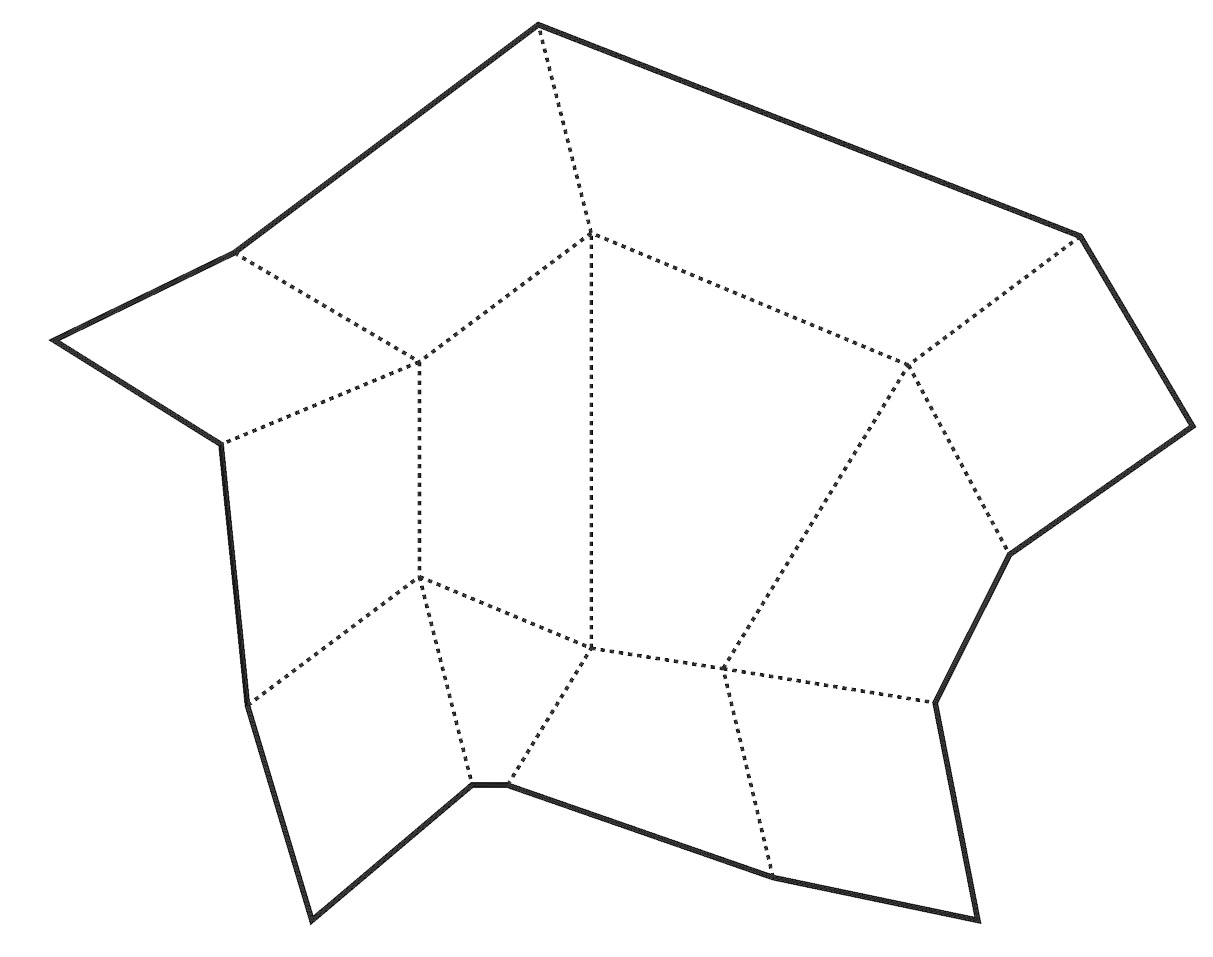}
            \includegraphics[width=0.32\textwidth]{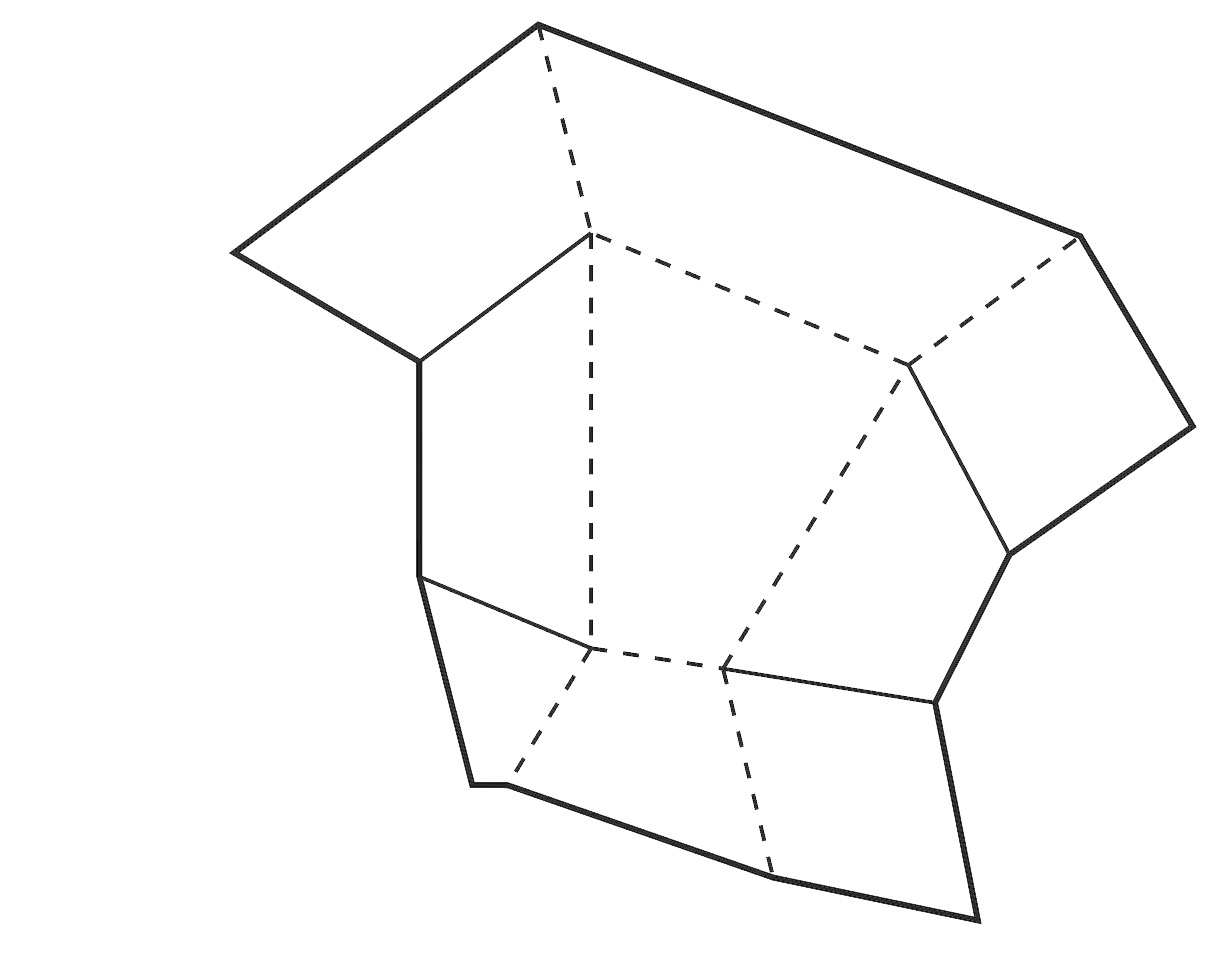}
      \includegraphics[width=0.7\textwidth]{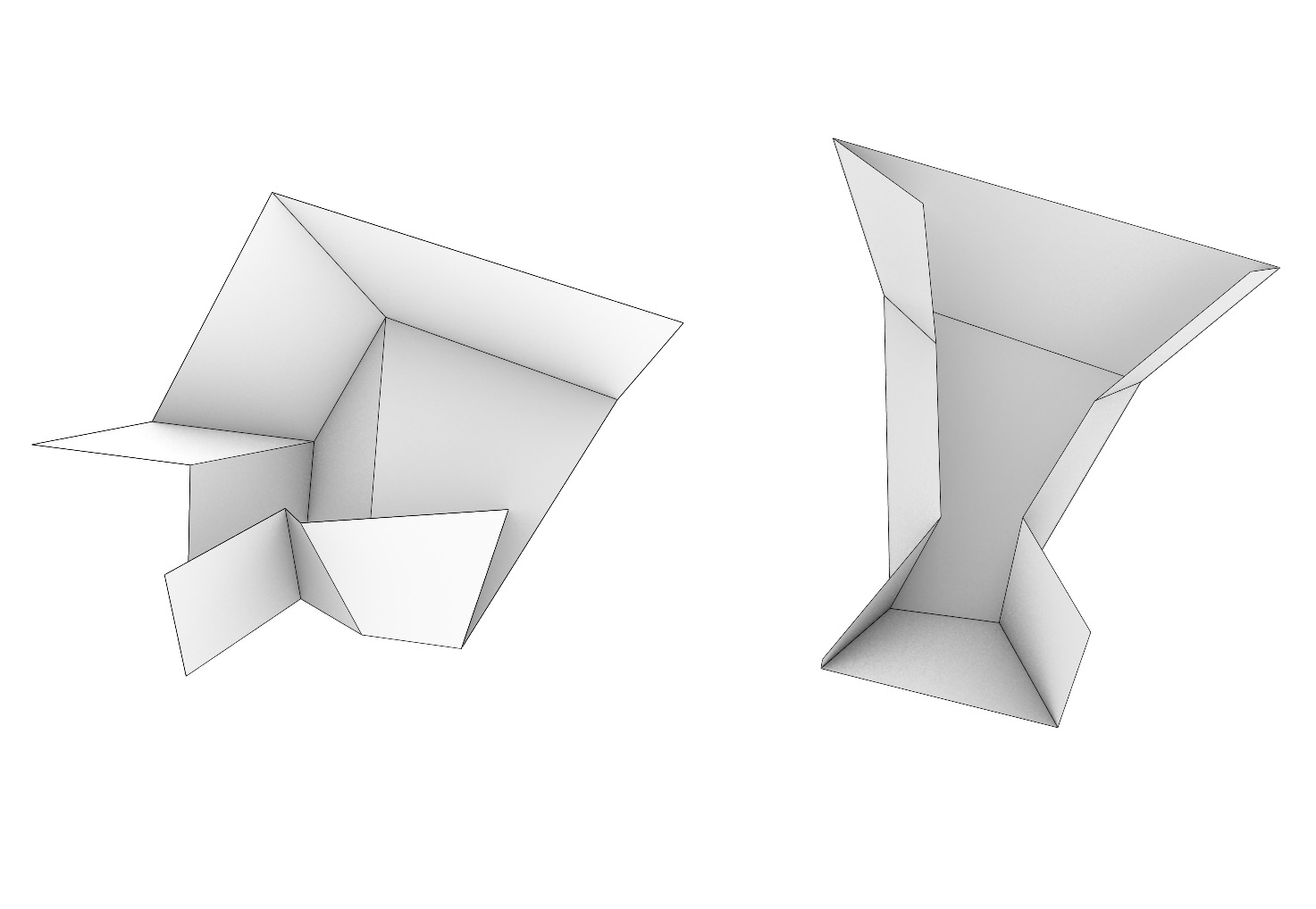}
\caption{Rigid planar $3 \times 4$ net (top middle) and deformations of its $3 \times 3$-subnets. In the top row left and right, dashed lines stand for valley folds, solid lines stand for mountain folds.}
    \label{fig:example-2}
\end{figure}

The angle values for the net from Figure \ref{fig:example-2} are given in Figure \ref{fig:Counterexample2}.
The number in each corner is the tangent of the half of the corresponding face angle.
For example, the angles at the bottom left vertex are
\[
2\arctan\frac32, \quad 2\arctan 2, \quad 2\arctan\frac23, \quad 2\arctan\frac12.
\]
Due to $\tan\frac{\pi - x}2 = \frac{1}{\tan\frac{x}2}$ reciprocal numbers correspond to complementary angles.
Thus the interior face on the left is a trapezoid, and the one on the right is a circular quadrilateral.
Besides, the opposite angles at every vertex complement each other to $\pi$.
% The fact that all half-tangents are rational implies that the complex can be realized with rational vertex coordinates.

\begin{figure}[ht]
\begin{center}
\begin{picture}(0,0)%
\includegraphics{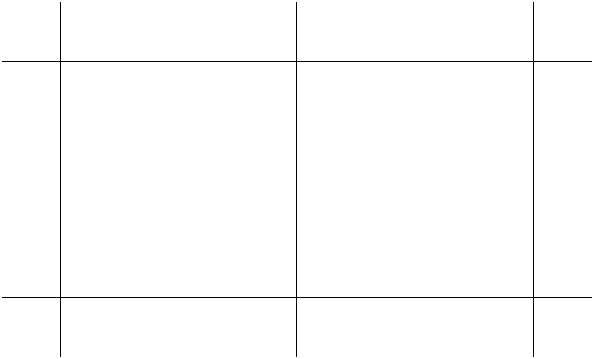}%
\end{picture}%
\setlength{\unitlength}{4144sp}%
\begin{picture}(4524,2724)(439,-2323)
\put(946,-2041){\makebox(0,0)[lb]{\smash{\fontsize{9}{10.8}\usefont{T1}{ptm}{m}{n}{\color[rgb]{0,0,0}$\frac12$}%
}}}
\put(766,-2041){\makebox(0,0)[lb]{\smash{\fontsize{9}{10.8}\usefont{T1}{ptm}{m}{n}{\color[rgb]{0,0,0}$\frac23$}%
}}}
\put(766,-1771){\makebox(0,0)[lb]{\smash{\fontsize{9}{10.8}\usefont{T1}{ptm}{m}{n}{\color[rgb]{0,0,0}$2$}%
}}}
\put(946,-1771){\makebox(0,0)[lb]{\smash{\fontsize{9}{10.8}\usefont{T1}{ptm}{m}{n}{\color[rgb]{0,0,0}$\frac32$}%
}}}
\put(2746,-2041){\makebox(0,0)[lb]{\smash{\fontsize{9}{10.8}\usefont{T1}{ptm}{m}{n}{\color[rgb]{0,0,0}$\frac32$}%
}}}
\put(2746,-1771){\makebox(0,0)[lb]{\smash{\fontsize{9}{10.8}\usefont{T1}{ptm}{m}{n}{\color[rgb]{0,0,0}$\frac76$}%
}}}
\put(2566,-2041){\makebox(0,0)[lb]{\smash{\fontsize{9}{10.8}\usefont{T1}{ptm}{m}{n}{\color[rgb]{0,0,0}$\frac67$}%
}}}
\put(2566,-1771){\makebox(0,0)[lb]{\smash{\fontsize{9}{10.8}\usefont{T1}{ptm}{m}{n}{\color[rgb]{0,0,0}$\frac23$}%
}}}
\put(766, 29){\makebox(0,0)[lb]{\smash{\fontsize{9}{10.8}\usefont{T1}{ptm}{m}{n}{\color[rgb]{0,0,0}$\frac12$}%
}}}
\put(766,-241){\makebox(0,0)[lb]{\smash{\fontsize{9}{10.8}\usefont{T1}{ptm}{m}{n}{\color[rgb]{0,0,0}$\frac23$}%
}}}
\put(946,-241){\makebox(0,0)[lb]{\smash{\fontsize{9}{10.8}\usefont{T1}{ptm}{m}{n}{\color[rgb]{0,0,0}$2$}%
}}}
\put(946, 29){\makebox(0,0)[lb]{\smash{\fontsize{9}{10.8}\usefont{T1}{ptm}{m}{n}{\color[rgb]{0,0,0}$\frac32$}%
}}}
\put(2566,-241){\makebox(0,0)[lb]{\smash{\fontsize{9}{10.8}\usefont{T1}{ptm}{m}{n}{\color[rgb]{0,0,0}$\frac12$}%
}}}
\put(2566, 29){\makebox(0,0)[lb]{\smash{\fontsize{9}{10.8}\usefont{T1}{ptm}{m}{n}{\color[rgb]{0,0,0}$\frac32$}%
}}}
\put(2746,-241){\makebox(0,0)[lb]{\smash{\fontsize{9}{10.8}\usefont{T1}{ptm}{m}{n}{\color[rgb]{0,0,0}$\frac23$}%
}}}
\put(2746, 29){\makebox(0,0)[lb]{\smash{\fontsize{9}{10.8}\usefont{T1}{ptm}{m}{n}{\color[rgb]{0,0,0}$2$}%
}}}
\put(4546,-2041){\makebox(0,0)[lb]{\smash{\fontsize{9}{10.8}\usefont{T1}{ptm}{m}{n}{\color[rgb]{0,0,0}$\frac23$}%
}}}
\put(4546,-1771){\makebox(0,0)[lb]{\smash{\fontsize{9}{10.8}\usefont{T1}{ptm}{m}{n}{\color[rgb]{0,0,0}$\frac{51}{76}$}%
}}}
\put(4366,-1771){\makebox(0,0)[lb]{\smash{\fontsize{9}{10.8}\usefont{T1}{ptm}{m}{n}{\color[rgb]{0,0,0}$\frac32$}%
}}}
\put(4366,-241){\makebox(0,0)[lb]{\smash{\fontsize{9}{10.8}\usefont{T1}{ptm}{m}{n}{\color[rgb]{0,0,0}$\frac67$}%
}}}
\put(4546,-241){\makebox(0,0)[lb]{\smash{\fontsize{9}{10.8}\usefont{T1}{ptm}{m}{n}{\color[rgb]{0,0,0}$\frac47$}%
}}}
\put(4546, 29){\makebox(0,0)[lb]{\smash{\fontsize{9}{10.8}\usefont{T1}{ptm}{m}{n}{\color[rgb]{0,0,0}$\frac76$}%
}}}
\put(4366, 29){\makebox(0,0)[lb]{\smash{\fontsize{9}{10.8}\usefont{T1}{ptm}{m}{n}{\color[rgb]{0,0,0}$\frac74$}%
}}}
\put(4321,-2041){\makebox(0,0)[lb]{\smash{\fontsize{9}{10.8}\usefont{T1}{ptm}{m}{n}{\color[rgb]{0,0,0}$\frac{76}{51}$}%
}}}
\end{picture}%
\end{center}
\caption{A planar non-flexible $3\times 4$-complex comprising two flexible $3 \times 3$-subcomplexes.}
\label{fig:Counterexample2}
\end{figure}

\begin{thm}
\label{thm:Counterexample2}
The planar discrete conjugate net depicted in Figure \ref{fig:Counterexample2} is not isometrically deformable while both of its $3 \times 3$-subnets are.
\end{thm}

In order to prove this theorem, let us study the configuration spaces of the interior vertices of our net.
% Since the opposite angles at every vertex complement each other to $\pi$, the dependence between the dihedral angles especially simple.

\begin{lem}
\label{lem:ConfSpace}
Let the sides of the spherical quadrilateral in Figure \ref{fig:2x2TangentImage}, right, satisfy $\gamma = \pi - \alpha, \delta = \pi - \beta$.
Introduce the notations
\begin{gather*}
\tan\frac{\alpha}2 = a, \quad \tan\frac{\beta}2 = b, \quad \tan\frac{\gamma}2 = c, \quad
\tan\frac{\delta}2 = d,\\
\tan\frac{\phi_+}2 = z_+, \quad
\tan\frac{\psi_+}2 = w_+, \quad
\tan\frac{\phi_-}2 = z_-, \quad
\tan\frac{\psi_-}2 = w_-.
\end{gather*}
Further, assume that the smallest among the angles $\alpha, \beta, \gamma, \delta$ is unique, and this angle is $\alpha$.
Then the configuration space of the spherical quadrilateral consists of two components:
\[
z_+ = -t, \quad z_- = t, \quad
w_+ = \frac{1+ad}{1-ad} t, \quad w_- = \frac{1+ad}{1-ad}t
\]
and
\[
z_+ = \frac{1+ab}{1-ab}t, \quad z_- = \frac{1+ab}{1-ab}t, \quad
w_+ = -t, \quad w_- = t.
\]
\end{lem}

\begin{proof}
Replace the common vertex of the edges $\gamma$ and $\delta$ by its antipode.
The edges become replaced by their complements $\pi-\gamma = \alpha$ and $\pi-\delta = \beta$, and the quadrilateral turns into an isogram (a non-self-intersecting quadrilateral with equal pairs of opposite sides) or an anti-isogram (a self-intersecting quadrilateral with equal pairs of opposite sides).
The first case is pictured in Figure \ref{fig:Parallelograms}, right.
In the second case one obtains an isogram if the $\beta\gamma$-vertex is replaced by its antipode, Figure \ref{fig:Parallelograms}, left.
Figure \ref{fig:Parallelograms} describes all possible configurations up to orientation change: instead of three positive and one negative exterior angle there can be three negative and one positive one.

Now Lemma follows from the Lemma \ref{lem:Napier} below.

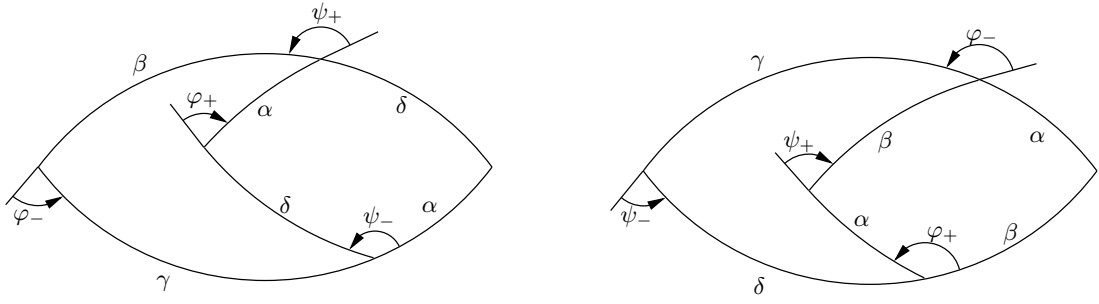
\begin{figure}[ht]
\begin{center}
\begin{picture}(0,0)%
\includegraphics{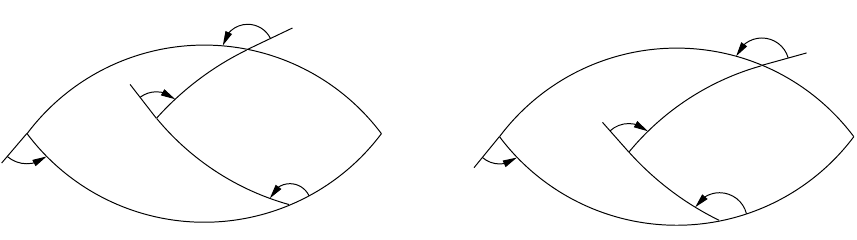}%
\end{picture}%
\setlength{\unitlength}{4144sp}%
\begin{picture}(6514,1819)(-205,709)
\put(561,2088){\makebox(0,0)[lb]{\smash{\fontsize{9}{10.8}\usefont{T1}{ptm}{m}{n}{\color[rgb]{0,0,0}$\beta$}%
}}}
\put(1437,1245){\makebox(0,0)[lb]{\smash{\fontsize{9}{10.8}\usefont{T1}{ptm}{m}{n}{\color[rgb]{0,0,0}$\delta$}%
}}}
\put(2129,1837){\makebox(0,0)[lb]{\smash{\fontsize{9}{10.8}\usefont{T1}{ptm}{m}{n}{\color[rgb]{0,0,0}$\delta$}%
}}}
\put(884,1900){\makebox(0,0)[lb]{\smash{\fontsize{9}{10.8}\usefont{T1}{ptm}{m}{n}{\color[rgb]{0,0,0}$\phi_+$}%
}}}
\put(704,809){\makebox(0,0)[lb]{\smash{\fontsize{9}{10.8}\usefont{T1}{ptm}{m}{n}{\color[rgb]{0,0,0}$\gamma$}%
}}}
\put(-149,1200){\makebox(0,0)[lb]{\smash{\fontsize{9}{10.8}\usefont{T1}{ptm}{m}{n}{\color[rgb]{0,0,0}$\phi_-$}%
}}}
\put(1934,1182){\makebox(0,0)[lb]{\smash{\fontsize{9}{10.8}\usefont{T1}{ptm}{m}{n}{\color[rgb]{0,0,0}$\psi_-$}%
}}}
\put(1626,2400){\makebox(0,0)[lb]{\smash{\fontsize{9}{10.8}\usefont{T1}{ptm}{m}{n}{\color[rgb]{0,0,0}$\psi_+$}%
}}}
\put(2283,1248){\makebox(0,0)[lb]{\smash{\fontsize{9}{10.8}\usefont{T1}{ptm}{m}{n}{\color[rgb]{0,0,0}$\alpha$}%
}}}
\put(1303,1820){\makebox(0,0)[lb]{\smash{\fontsize{9}{10.8}\usefont{T1}{ptm}{m}{n}{\color[rgb]{0,0,0}$\alpha$}%
}}}
\put(4242,2119){\rotatebox{360.0}{\makebox(0,0)[lb]{\smash{\fontsize{9}{10.8}\usefont{T1}{ptm}{m}{n}{\color[rgb]{0,0,0}$\gamma$}%
}}}}
\put(4256,761){\rotatebox{360.0}{\makebox(0,0)[lb]{\smash{\fontsize{9}{10.8}\usefont{T1}{ptm}{m}{n}{\color[rgb]{0,0,0}$\delta$}%
}}}}
\put(3466,1188){\rotatebox{360.0}{\makebox(0,0)[lb]{\smash{\fontsize{9}{10.8}\usefont{T1}{ptm}{m}{n}{\color[rgb]{0,0,0}$\psi_-$}%
}}}}
\put(4436,1636){\rotatebox{360.0}{\makebox(0,0)[lb]{\smash{\fontsize{9}{10.8}\usefont{T1}{ptm}{m}{n}{\color[rgb]{0,0,0}$\psi_+$}%
}}}}
\put(5521,2274){\rotatebox{360.0}{\makebox(0,0)[lb]{\smash{\fontsize{9}{10.8}\usefont{T1}{ptm}{m}{n}{\color[rgb]{0,0,0}$\phi_-$}%
}}}}
\put(4855,1163){\rotatebox{360.0}{\makebox(0,0)[lb]{\smash{\fontsize{9}{10.8}\usefont{T1}{ptm}{m}{n}{\color[rgb]{0,0,0}$\alpha$}%
}}}}
\put(5288,1098){\rotatebox{360.0}{\makebox(0,0)[lb]{\smash{\fontsize{9}{10.8}\usefont{T1}{ptm}{m}{n}{\color[rgb]{0,0,0}$\phi_+$}%
}}}}
\put(5743,1083){\makebox(0,0)[lb]{\smash{\fontsize{9}{10.8}\usefont{T1}{ptm}{m}{n}{\color[rgb]{0,0,0}$\beta$}%
}}}
\put(4996,1630){\rotatebox{360.0}{\makebox(0,0)[lb]{\smash{\fontsize{9}{10.8}\usefont{T1}{ptm}{m}{n}{\color[rgb]{0,0,0}$\beta$}%
}}}}
\put(5900,1668){\rotatebox{360.0}{\makebox(0,0)[lb]{\smash{\fontsize{9}{10.8}\usefont{T1}{ptm}{m}{n}{\color[rgb]{0,0,0}$\alpha$}%
}}}}
\end{picture}%
\end{center}
\caption{To the proof of Lemma \ref{lem:ConfSpace}.}
\label{fig:Parallelograms}
\end{figure}
\end{proof}

\begin{lem}
\label{lem:Napier}
In every spherical isogram with side lengths $\alpha, \beta$ and interior angles $\phi, \psi$ the following equation holds:
\begin{equation}
\label{eqn:Napier}
\tan\frac{\phi}2 \tan\frac{\psi}2 =
\frac{1 + \tan\frac{\alpha}{2}\tan\frac{\beta}{2}}{1 - \tan\frac{\alpha}{2}\tan\frac{\beta}{2}}.
\end{equation}
Vice versa, for any $\alpha, \beta \in (0,\pi)$ and $\phi, \psi \in (0,2\pi)$ satisfying the above equation there is an isogram with side lengths $a,b$ and angles $\gamma, \delta$.
\end{lem}
The formula of Lemma \ref{lem:Napier} is contained in \cite[\S 22]{GrafSauer1931} without proof and with the right hand side in a different form.
For a proof see \cite[Theorem 3.19 and Remark 3.20]{Izmestiev2026}.

Figure \ref{fig:Parametrization} represents the two components of the configuration space described in Lemma \ref{lem:ConfSpace} in a schematic way.
Instead of the angles (both planar and dihedral) we are writing tangents of their halves.
The dashed and the full lines correspond to the opposite signs of the dihedral angles.
Thus, dashed can stand for valleys and full can stand for mountains, or vice versa.

\begin{figure}[ht]
\begin{center}
\begin{picture}(0,0)%
\includegraphics{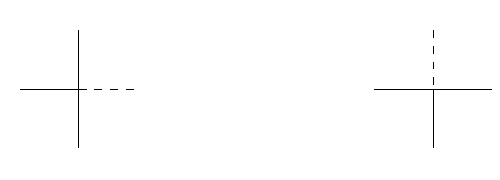}%
\end{picture}%
\setlength{\unitlength}{4144sp}%
\begin{picture}(3810,1352)(301,-747)
\put(3556,-646){\makebox(0,0)[lb]{\smash{\fontsize{9}{10.8}\usefont{T1}{ptm}{m}{n}{\color[rgb]{0,0,0}$t$}%
}}}
\put(3511,434){\makebox(0,0)[lb]{\smash{\fontsize{9}{10.8}\usefont{T1}{ptm}{m}{n}{\color[rgb]{0,0,0}$-t$}%
}}}
\put(316,-106){\makebox(0,0)[lb]{\smash{\fontsize{9}{10.8}\usefont{T1}{ptm}{m}{n}{\color[rgb]{0,0,0}$t$}%
}}}
\put(766,479){\makebox(0,0)[lb]{\smash{\fontsize{9}{10.8}\usefont{T1}{ptm}{m}{n}{\color[rgb]{0,0,0}$\frac{1+ad}{1-ad}t$}%
}}}
\put(721,-691){\makebox(0,0)[lb]{\smash{\fontsize{9}{10.8}\usefont{T1}{ptm}{m}{n}{\color[rgb]{0,0,0}$\frac{1+ad}{1-ad}t$}%
}}}
\put(1351,-106){\makebox(0,0)[lb]{\smash{\fontsize{9}{10.8}\usefont{T1}{ptm}{m}{n}{\color[rgb]{0,0,0}$(-t)$}%
}}}
\put(2746,-106){\makebox(0,0)[lb]{\smash{\fontsize{9}{10.8}\usefont{T1}{ptm}{m}{n}{\color[rgb]{0,0,0}$\frac{1+ab}{1-ab}t$}%
}}}
\put(4096,-106){\makebox(0,0)[lb]{\smash{\fontsize{9}{10.8}\usefont{T1}{ptm}{m}{n}{\color[rgb]{0,0,0}$\frac{1+ab}{1-ab}t$}%
}}}
\put(766, 29){\makebox(0,0)[lb]{\smash{\fontsize{9}{10.8}\usefont{T1}{ptm}{m}{n}{\color[rgb]{0,0,0}$b$}%
}}}
\put(766,-196){\makebox(0,0)[lb]{\smash{\fontsize{9}{10.8}\usefont{T1}{ptm}{m}{n}{\color[rgb]{0,0,0}$c$}%
}}}
\put(991, 29){\makebox(0,0)[lb]{\smash{\fontsize{9}{10.8}\usefont{T1}{ptm}{m}{n}{\color[rgb]{0,0,0}$a$}%
}}}
\put(991,-196){\makebox(0,0)[lb]{\smash{\fontsize{9}{10.8}\usefont{T1}{ptm}{m}{n}{\color[rgb]{0,0,0}$d$}%
}}}
\put(3466, 29){\makebox(0,0)[lb]{\smash{\fontsize{9}{10.8}\usefont{T1}{ptm}{m}{n}{\color[rgb]{0,0,0}$b$}%
}}}
\put(3466,-196){\makebox(0,0)[lb]{\smash{\fontsize{9}{10.8}\usefont{T1}{ptm}{m}{n}{\color[rgb]{0,0,0}$c$}%
}}}
\put(3691, 29){\makebox(0,0)[lb]{\smash{\fontsize{9}{10.8}\usefont{T1}{ptm}{m}{n}{\color[rgb]{0,0,0}$a$}%
}}}
\put(3691,-196){\makebox(0,0)[lb]{\smash{\fontsize{9}{10.8}\usefont{T1}{ptm}{m}{n}{\color[rgb]{0,0,0}$d$}%
}}}
\end{picture}%
\end{center}
\caption{Parametrization of the configuration space of a spherical quadrilateral with $\alpha + \gamma = \beta + \delta = \pi$.}
\label{fig:Parametrization}
\end{figure}
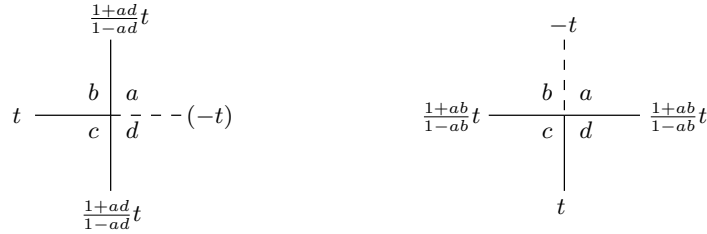

We are now in a position to prove Theorem \ref{thm:Counterexample2}.

\begin{proof}[Proof of Theorem \ref{thm:Counterexample2}]
The isometric deformations of the left and of the right $3 \times 3$-subnets are shown in Figure \ref{fig:Deformations}.
The full and the dashed edges represent mountains and valleys, and the numbers at the edges are the tangents of half the dihedral angles.
These are in accordance with Lemma \ref{lem:ConfSpace}, up to a parameter scaling at every vertex, see also Figure \ref{fig:Parametrization}.

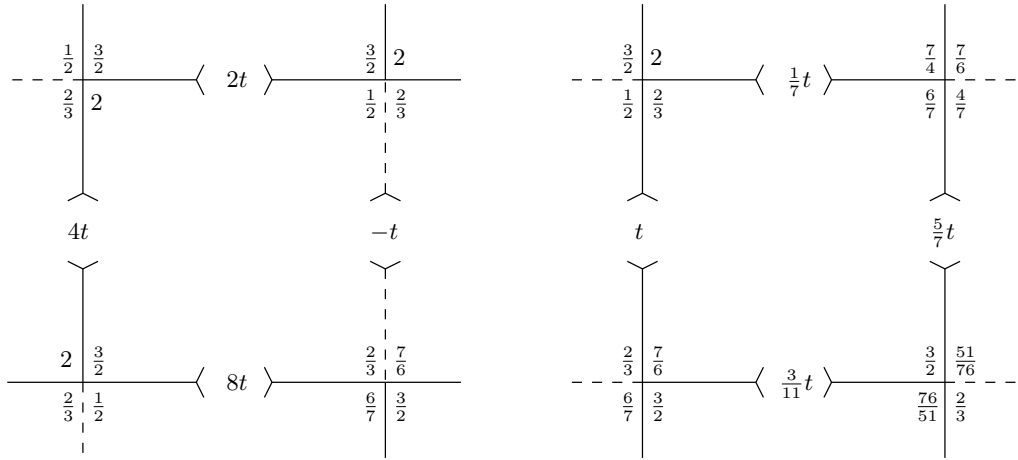
\begin{figure}[ht]
\begin{center}
\begin{picture}(0,0)%
\includegraphics{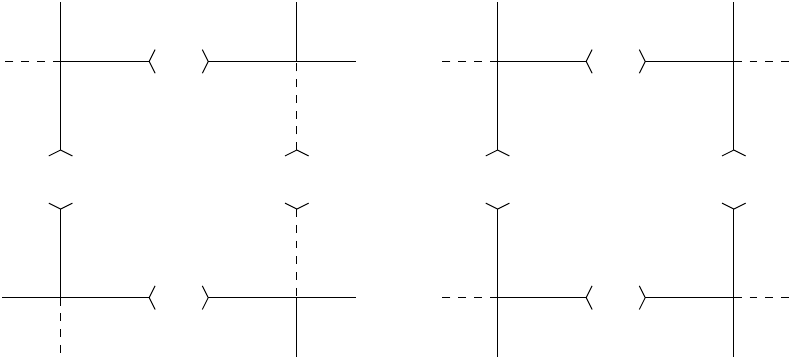}%
\end{picture}%
\setlength{\unitlength}{4144sp}%
\begin{picture}(6054,2724)(439,-2323)
\put(946,-2041){\makebox(0,0)[lb]{\smash{\fontsize{9}{10.8}\usefont{T1}{ptm}{m}{n}{\color[rgb]{0,0,0}$\frac12$}%
}}}
\put(766,-2041){\makebox(0,0)[lb]{\smash{\fontsize{9}{10.8}\usefont{T1}{ptm}{m}{n}{\color[rgb]{0,0,0}$\frac23$}%
}}}
\put(766,-1771){\makebox(0,0)[lb]{\smash{\fontsize{9}{10.8}\usefont{T1}{ptm}{m}{n}{\color[rgb]{0,0,0}$2$}%
}}}
\put(946,-1771){\makebox(0,0)[lb]{\smash{\fontsize{9}{10.8}\usefont{T1}{ptm}{m}{n}{\color[rgb]{0,0,0}$\frac32$}%
}}}
\put(2746,-2041){\makebox(0,0)[lb]{\smash{\fontsize{9}{10.8}\usefont{T1}{ptm}{m}{n}{\color[rgb]{0,0,0}$\frac32$}%
}}}
\put(2746,-1771){\makebox(0,0)[lb]{\smash{\fontsize{9}{10.8}\usefont{T1}{ptm}{m}{n}{\color[rgb]{0,0,0}$\frac76$}%
}}}
\put(2566,-2041){\makebox(0,0)[lb]{\smash{\fontsize{9}{10.8}\usefont{T1}{ptm}{m}{n}{\color[rgb]{0,0,0}$\frac67$}%
}}}
\put(2566,-1771){\makebox(0,0)[lb]{\smash{\fontsize{9}{10.8}\usefont{T1}{ptm}{m}{n}{\color[rgb]{0,0,0}$\frac23$}%
}}}
\put(766, 29){\makebox(0,0)[lb]{\smash{\fontsize{9}{10.8}\usefont{T1}{ptm}{m}{n}{\color[rgb]{0,0,0}$\frac12$}%
}}}
\put(766,-241){\makebox(0,0)[lb]{\smash{\fontsize{9}{10.8}\usefont{T1}{ptm}{m}{n}{\color[rgb]{0,0,0}$\frac23$}%
}}}
\put(946,-241){\makebox(0,0)[lb]{\smash{\fontsize{9}{10.8}\usefont{T1}{ptm}{m}{n}{\color[rgb]{0,0,0}$2$}%
}}}
\put(946, 29){\makebox(0,0)[lb]{\smash{\fontsize{9}{10.8}\usefont{T1}{ptm}{m}{n}{\color[rgb]{0,0,0}$\frac32$}%
}}}
\put(2566,-241){\makebox(0,0)[lb]{\smash{\fontsize{9}{10.8}\usefont{T1}{ptm}{m}{n}{\color[rgb]{0,0,0}$\frac12$}%
}}}
\put(2566, 29){\makebox(0,0)[lb]{\smash{\fontsize{9}{10.8}\usefont{T1}{ptm}{m}{n}{\color[rgb]{0,0,0}$\frac32$}%
}}}
\put(2746,-241){\makebox(0,0)[lb]{\smash{\fontsize{9}{10.8}\usefont{T1}{ptm}{m}{n}{\color[rgb]{0,0,0}$\frac23$}%
}}}
\put(2746, 29){\makebox(0,0)[lb]{\smash{\fontsize{9}{10.8}\usefont{T1}{ptm}{m}{n}{\color[rgb]{0,0,0}$2$}%
}}}
\put(1756,-106){\makebox(0,0)[lb]{\smash{\fontsize{9}{10.8}\usefont{T1}{ptm}{m}{n}{\color[rgb]{0,0,0}$2t$}%
}}}
\put(1756,-1906){\makebox(0,0)[lb]{\smash{\fontsize{9}{10.8}\usefont{T1}{ptm}{m}{n}{\color[rgb]{0,0,0}$8t$}%
}}}
\put(811,-1006){\makebox(0,0)[lb]{\smash{\fontsize{9}{10.8}\usefont{T1}{ptm}{m}{n}{\color[rgb]{0,0,0}$4t$}%
}}}
\put(2611,-1006){\makebox(0,0)[lb]{\smash{\fontsize{9}{10.8}\usefont{T1}{ptm}{m}{n}{\color[rgb]{0,0,0}$-t$}%
}}}
\put(4276,-2041){\makebox(0,0)[lb]{\smash{\fontsize{9}{10.8}\usefont{T1}{ptm}{m}{n}{\color[rgb]{0,0,0}$\frac32$}%
}}}
\put(4096,-2041){\makebox(0,0)[lb]{\smash{\fontsize{9}{10.8}\usefont{T1}{ptm}{m}{n}{\color[rgb]{0,0,0}$\frac67$}%
}}}
\put(4096,-1771){\makebox(0,0)[lb]{\smash{\fontsize{9}{10.8}\usefont{T1}{ptm}{m}{n}{\color[rgb]{0,0,0}$\frac23$}%
}}}
\put(4276,-1771){\makebox(0,0)[lb]{\smash{\fontsize{9}{10.8}\usefont{T1}{ptm}{m}{n}{\color[rgb]{0,0,0}$\frac76$}%
}}}
\put(6076,-2041){\makebox(0,0)[lb]{\smash{\fontsize{9}{10.8}\usefont{T1}{ptm}{m}{n}{\color[rgb]{0,0,0}$\frac23$}%
}}}
\put(6076,-1771){\makebox(0,0)[lb]{\smash{\fontsize{9}{10.8}\usefont{T1}{ptm}{m}{n}{\color[rgb]{0,0,0}$\frac{51}{76}$}%
}}}
\put(5896,-1771){\makebox(0,0)[lb]{\smash{\fontsize{9}{10.8}\usefont{T1}{ptm}{m}{n}{\color[rgb]{0,0,0}$\frac32$}%
}}}
\put(4096, 29){\makebox(0,0)[lb]{\smash{\fontsize{9}{10.8}\usefont{T1}{ptm}{m}{n}{\color[rgb]{0,0,0}$\frac32$}%
}}}
\put(4276,-241){\makebox(0,0)[lb]{\smash{\fontsize{9}{10.8}\usefont{T1}{ptm}{m}{n}{\color[rgb]{0,0,0}$\frac23$}%
}}}
\put(4276, 29){\makebox(0,0)[lb]{\smash{\fontsize{9}{10.8}\usefont{T1}{ptm}{m}{n}{\color[rgb]{0,0,0}$2$}%
}}}
\put(5896,-241){\makebox(0,0)[lb]{\smash{\fontsize{9}{10.8}\usefont{T1}{ptm}{m}{n}{\color[rgb]{0,0,0}$\frac67$}%
}}}
\put(6076,-241){\makebox(0,0)[lb]{\smash{\fontsize{9}{10.8}\usefont{T1}{ptm}{m}{n}{\color[rgb]{0,0,0}$\frac47$}%
}}}
\put(6076, 29){\makebox(0,0)[lb]{\smash{\fontsize{9}{10.8}\usefont{T1}{ptm}{m}{n}{\color[rgb]{0,0,0}$\frac76$}%
}}}
\put(5086,-106){\makebox(0,0)[lb]{\smash{\fontsize{9}{10.8}\usefont{T1}{ptm}{m}{n}{\color[rgb]{0,0,0}$\frac{1}{7}t$}%
}}}
\put(5941,-1006){\makebox(0,0)[lb]{\smash{\fontsize{9}{10.8}\usefont{T1}{ptm}{m}{n}{\color[rgb]{0,0,0}$\frac{5}{7}t$}%
}}}
\put(4096,-241){\makebox(0,0)[lb]{\smash{\fontsize{9}{10.8}\usefont{T1}{ptm}{m}{n}{\color[rgb]{0,0,0}$\frac12$}%
}}}
\put(4186,-1006){\makebox(0,0)[lb]{\smash{\fontsize{9}{10.8}\usefont{T1}{ptm}{m}{n}{\color[rgb]{0,0,0}$t$}%
}}}
\put(5896, 29){\makebox(0,0)[lb]{\smash{\fontsize{9}{10.8}\usefont{T1}{ptm}{m}{n}{\color[rgb]{0,0,0}$\frac74$}%
}}}
\put(5041,-1906){\makebox(0,0)[lb]{\smash{\fontsize{9}{10.8}\usefont{T1}{ptm}{m}{n}{\color[rgb]{0,0,0}$\frac{3}{11}t$}%
}}}
\put(5851,-2041){\makebox(0,0)[lb]{\smash{\fontsize{9}{10.8}\usefont{T1}{ptm}{m}{n}{\color[rgb]{0,0,0}$\frac{76}{51}$}%
}}}
\end{picture}%
\end{center}
\caption{Isometric deformations of the left and the right subcomplexes.}
\label{fig:Deformations}
\end{figure}

The deformations in Figure \ref{fig:Deformations} do not agree at the common pair of vertices, even after any parameter change.
Thus they do not fit together to result in an isometric deformation of the $3 \times 4$-net.
In order to prove the rigidity of the latter, it suffices to show that the deformations of the $3 \times 3$-subnets presented in Figure \ref{fig:Deformations} are unique, up to parameter change.
In other words, one wants to show that there is no other choice of components of the configuration spaces at the vertices that would fit together to form a deformation of a $3 \times 3$-subnet.
For this, we write down all scaling factors between the tangents of the halves of the dihedral angles, see Figure \ref{fig:Combinations}.
The products of the scaling factors when going from one edge to its opposite in two different ways must be equal.
Only one combination for each of the two $3\times 3$-complexes satisfies this condition; these factors in Figure \ref{fig:Combinations} are boxed.
This confirms the uniqueness of the deformations shown in Figure \ref{fig:Deformations}, and the theorem is proved.
\end{proof}

\begin{figure}[ht]
\begin{center}
\begin{picture}(0,0)%
\includegraphics{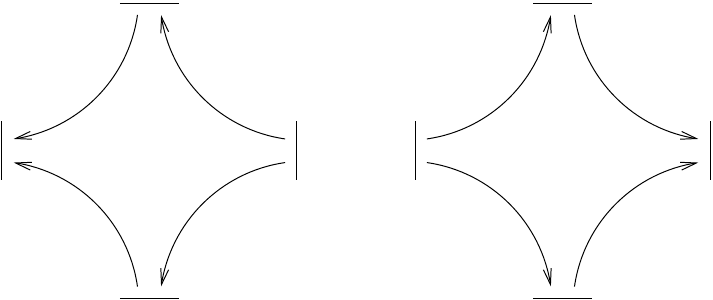}%
\end{picture}%
\setlength{\unitlength}{4144sp}%
\begin{picture}(5424,2274)(-11,-1423)
\put( 91,344){\makebox(0,0)[lb]{\smash{\fontsize{9}{10.8}\usefont{T1}{ptm}{m}{n}{\color[rgb]{0,0,0}$\boxed{2} \text{ or } \frac17$}%
}}}
\put( 91,-916){\makebox(0,0)[lb]{\smash{\fontsize{9}{10.8}\usefont{T1}{ptm}{m}{n}{\color[rgb]{0,0,0}$7 \text{ or } \boxed{\frac12}$}%
}}}
\put(1576,344){\makebox(0,0)[lb]{\smash{\fontsize{9}{10.8}\usefont{T1}{ptm}{m}{n}{\color[rgb]{0,0,0}$\boxed{2} \text{ or } \frac17$}%
}}}
\put(1576,-916){\makebox(0,0)[lb]{\smash{\fontsize{9}{10.8}\usefont{T1}{ptm}{m}{n}{\color[rgb]{0,0,0}$\boxed{8} \text{ or } \frac{3}{11}$}%
}}}
\put(3196,344){\makebox(0,0)[lb]{\smash{\fontsize{9}{10.8}\usefont{T1}{ptm}{m}{n}{\color[rgb]{0,0,0}$2 \text{ or } \boxed{\frac17}$}%
}}}
\put(3196,-916){\makebox(0,0)[lb]{\smash{\fontsize{9}{10.8}\usefont{T1}{ptm}{m}{n}{\color[rgb]{0,0,0}$8 \text{ or } \boxed{\frac{3}{11}}$}%
}}}
\put(4726,344){\makebox(0,0)[lb]{\smash{\fontsize{9}{10.8}\usefont{T1}{ptm}{m}{n}{\color[rgb]{0,0,0}$\frac{25}{73} \text{ or } \boxed{5}$}%
}}}
\put(4771,-916){\makebox(0,0)[lb]{\smash{\fontsize{9}{10.8}\usefont{T1}{ptm}{m}{n}{\color[rgb]{0,0,0}$\boxed{\frac{55}{21}} \text{ or } \frac{1}{305}$}%
}}}
\end{picture}%
\end{center}
\caption{Proving the uniqueness of the deformations shown in Figure \ref{fig:Parametrization}.}
\label{fig:Combinations}
\end{figure}
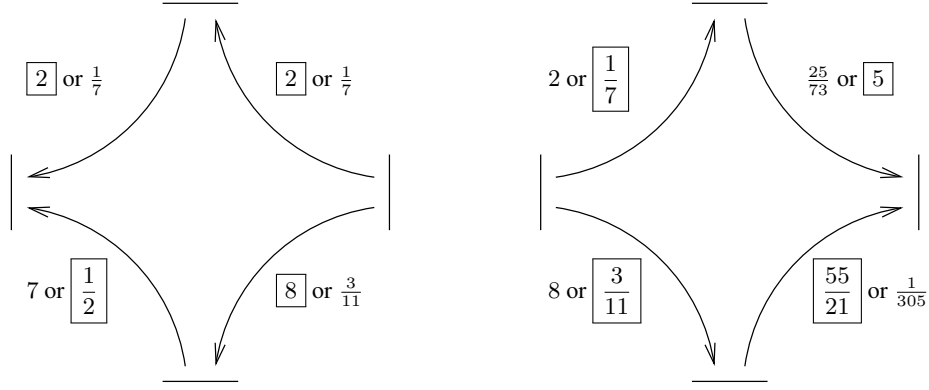

\section{Infinitesimal isometric deformations}
\label{sec:IID}
An infinitesimal isometry of $\RR^3$ is a vector field of the form $\xi(x) = Ax+b$ with a skew-symmetric matrix~$A$ or, equivalently, $\xi(x) = r \times x + b$.
Intuitively this is an infinitesimal translation ( if $r = 0$) or an infinitesimal screw motion (if $r \ne 0$) about an axis parallel to $r$.
An infinitesimal isometric deformation of a polyhedral surface is an assignment of an infinitesimal isometry to every face of the surface such that their restrictions to common edges agree.
An infinitesimal isometric deformation is called trivial if all isometries of the faces are the same.
A polyhedral surface is called infinitesimally flexible if it admits a non-trivial infinitesimal isometric deformation.

For any two faces $F_1$ and $F_2$ sharing an edge the difference of the respective vector fields has to vanish on the edge.
This implies $\xi_1(x) - \xi_2(x) = r \times x$ with $r$ parallel to the common edge.
Summing these differences for pairs of adjacent faces around a vertex one arrives at the following classical lemma.

\begin{lem}
Let $F_1, F_2, \ldots, F_n$ be faces surrounding a vertex of a polyhedral surface, in this cyclic order, and let $r_i(x)$ be the vector of the relative rotation of the face $F_{i+1}$ with respect to the face $F_i$, with indices taken modulo $n$.
Then one has $\sum_{i=1}^n r_i = 0$.
\end{lem}

This leads to the following characterization of the infinitesimal flexibility of polyhedral surface.

\begin{thm}
\label{thm:Weights}
A simply-connected polyhedral surface is infinitesimally flexible if and only if there is an assignment of real weights $\omega_{ij}$ to its interior edges $p_ip_j$ such that for every interior vertex $p_i$ one has
\begin{equation}
\label{eqn:Weights}
\sum_{j} \omega_{ij}(p_i - p_j) = 0.
\end{equation}
\end{thm}

Equation \eqref{eqn:Weights} is equivalent to the existence of a reciprocal-parallel mesh of dual combinatorics.
The theorem goes back probably to the end of the XIX century, but we cannot provide an exact reference.

The globalization theorem for infinitesimal isometric deformations holds under the assumption that no flattening occurs.
Its proof is similar to that of Theorem \ref{thm:GlobFlex}.

\begin{thm}
\label{thm:GlobInfFlex}
If an $m \times n$ discrete conjugate net has no pair of coplanar adjacent faces and all of its $3 \times 3$-subnets are infinitesimally flexible, then the net is infinitesimally flexible as a whole.
\end{thm}
\begin{proof}
If there are no coplanar adjacent faces, then among the four edges adjacent to an interior vertex~$p_i$ no three are coplanar, and the equation \eqref{eqn:Weights} has a one-dimensional solution space spanned by a vector with non-zero components.
Choose a basic solution for each of the equations \eqref{eqn:Weights}.
In order to find a solution for the whole system of equations, one has to scale each of the individual solutions so that the values of $\omega_{ij}$ coming from two incident vertices agree.
An arbitrarily chosen scaling factor at one of the vertices propagates along any edge path in a unique way.
It suffices to show that the resulting value at any other vertex is independent of the choice of the path.
By a standard argument this holds if the propagation along the boundary of each square yields no contradiction.
This is equivalent to every $3 \times 3$-subnet being infinitesimally flexible, and the theorem is proved.
\end{proof}

\begin{rem}
Kokotsakis \cite{Kokotsakis1933} gave a characterization of the infinitesimal flexibility of a $3 \times 3$ discrete conjugate net in geometric terms, as collinearity of three points constructed by intersecting certain triples of planes.
\end{rem}

The following example shows that the globalization may fail in the presence of flattenings.

\begin{thm}
\label{thm:Counterexample3}
In the $3 \times 4$ discrete conjugate net depicted in Figures \ref{fig:CounterexampleInfRig0} and \ref{fig:CounterexampleInfRig} both of the $3 \times 3$-subnets are infinitesimally flexible while the net as a whole is infinitesimally rigid.
\end{thm}

\begin{figure}[ht]
\begin{center}
\includegraphics[width=.45\textwidth]{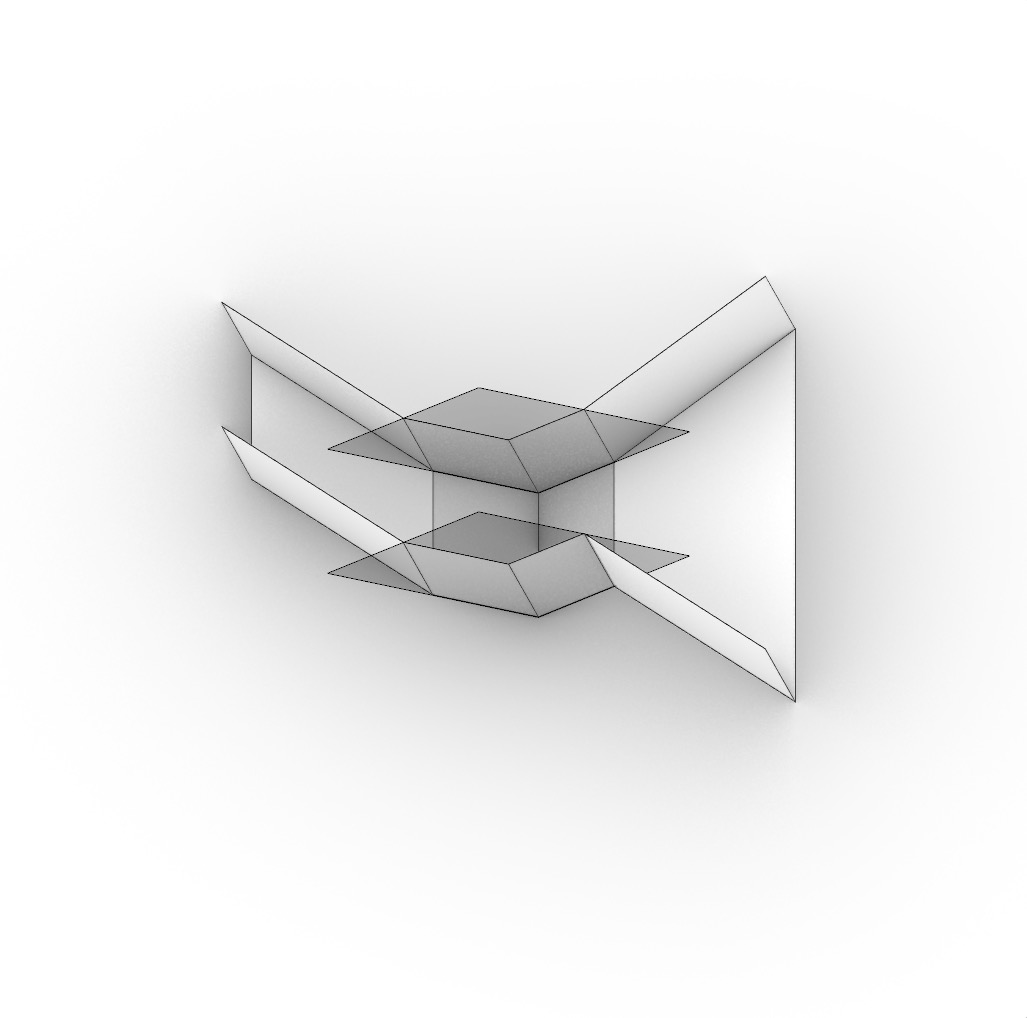}
\includegraphics[width=.45\textwidth]{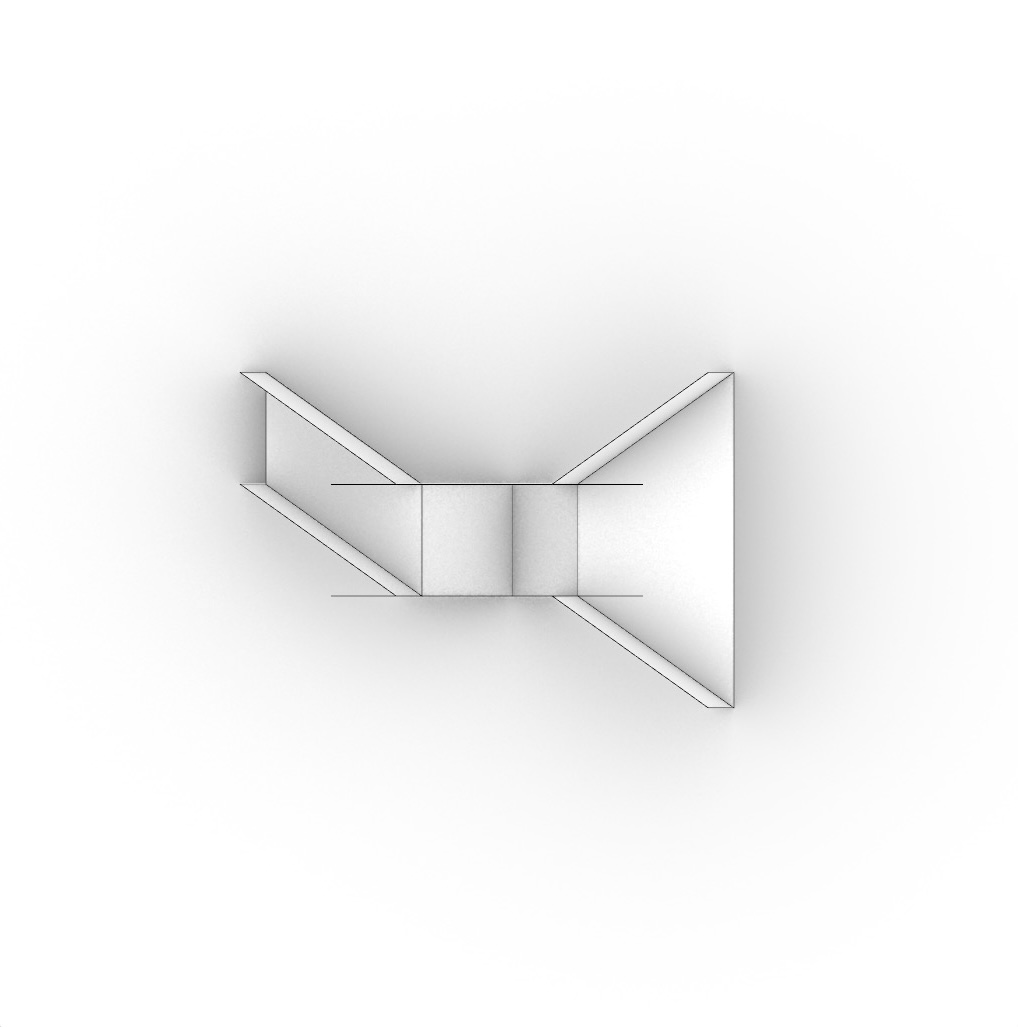}
\end{center}
\caption{An infinitesimally rigid $3 \times 4$-net with infinitesimally flexible $3 \times 3$-subnets.}
\label{fig:CounterexampleInfRig0}
\end{figure}

\begin{figure}[ht]
\begin{center}
\begin{picture}(0,0)%
\includegraphics{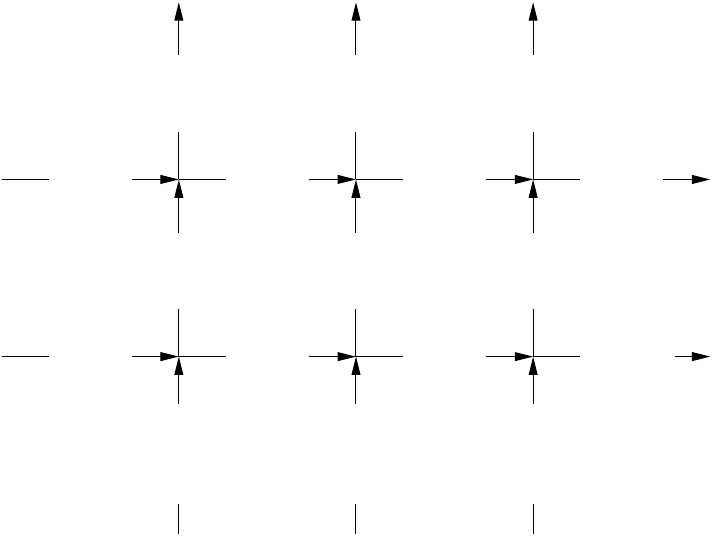}%
\end{picture}%
\setlength{\unitlength}{4144sp}%
\begin{picture}(5424,4074)(-11,-3223)
\put(406,-1906){\makebox(0,0)[lb]{\smash{\fontsize{9}{10.8}\usefont{T1}{ptm}{m}{n}{\color[rgb]{0,0,0}$(1,1,-1)$}%
}}}
\put(1396,-2941){\rotatebox{90.0}{\makebox(0,0)[lb]{\smash{\fontsize{9}{10.8}\usefont{T1}{ptm}{m}{n}{\color[rgb]{0,0,0}$(0,-1,-1)$}%
}}}}
\put(2746,-2941){\rotatebox{90.0}{\makebox(0,0)[lb]{\smash{\fontsize{9}{10.8}\usefont{T1}{ptm}{m}{n}{\color[rgb]{0,0,0}$(0,-1,-1)$}%
}}}}
\put(1756,-1906){\makebox(0,0)[lb]{\smash{\fontsize{9}{10.8}\usefont{T1}{ptm}{m}{n}{\color[rgb]{0,0,0}$(0,0,-1)$}%
}}}
\put(4096,-2941){\rotatebox{90.0}{\makebox(0,0)[lb]{\smash{\fontsize{9}{10.8}\usefont{T1}{ptm}{m}{n}{\color[rgb]{0,0,0}$(0,-1,-1)$}%
}}}}
\put(406,-556){\makebox(0,0)[lb]{\smash{\fontsize{9}{10.8}\usefont{T1}{ptm}{m}{n}{\color[rgb]{0,0,0}$(1,1,-1)$}%
}}}
\put(1756,-556){\makebox(0,0)[lb]{\smash{\fontsize{9}{10.8}\usefont{T1}{ptm}{m}{n}{\color[rgb]{0,0,0}$(0,0,-1)$}%
}}}
\put(4456,-556){\makebox(0,0)[lb]{\smash{\fontsize{9}{10.8}\usefont{T1}{ptm}{m}{n}{\color[rgb]{0,0,0}$(1,1,-1)$}%
}}}
\put(1396,-106){\rotatebox{90.0}{\makebox(0,0)[lb]{\smash{\fontsize{9}{10.8}\usefont{T1}{ptm}{m}{n}{\color[rgb]{0,0,0}$(0,1,1)$}%
}}}}
\put(2746,-106){\rotatebox{90.0}{\makebox(0,0)[lb]{\smash{\fontsize{9}{10.8}\usefont{T1}{ptm}{m}{n}{\color[rgb]{0,0,0}$(0,1,1)$}%
}}}}
\put(4096,-106){\rotatebox{90.0}{\makebox(0,0)[lb]{\smash{\fontsize{9}{10.8}\usefont{T1}{ptm}{m}{n}{\color[rgb]{0,0,0}$(0,1,1)$}%
}}}}
\put(3151,-556){\makebox(0,0)[lb]{\smash{\fontsize{9}{10.8}\usefont{T1}{ptm}{m}{n}{\color[rgb]{0,0,0}$(0,1,0)$}%
}}}
\put(3151,-1906){\makebox(0,0)[lb]{\smash{\fontsize{9}{10.8}\usefont{T1}{ptm}{m}{n}{\color[rgb]{0,0,0}$(0,1,0)$}%
}}}
\put(4434,-1906){\makebox(0,0)[lb]{\smash{\fontsize{9}{10.8}\usefont{T1}{ptm}{m}{n}{\color[rgb]{0,0,0}$(-1,1,-1)$}%
}}}
\put(1396,-1456){\rotatebox{90.0}{\makebox(0,0)[lb]{\smash{\fontsize{9}{10.8}\usefont{T1}{ptm}{m}{n}{\color[rgb]{0,0,0}$(1,0,0)$}%
}}}}
\put(2746,-1501){\rotatebox{90.0}{\makebox(0,0)[lb]{\smash{\fontsize{9}{10.8}\usefont{T1}{ptm}{m}{n}{\color[rgb]{0,0,0}$(1,0,0)$}%
}}}}
\put(4096,-1501){\rotatebox{90.0}{\makebox(0,0)[lb]{\smash{\fontsize{9}{10.8}\usefont{T1}{ptm}{m}{n}{\color[rgb]{0,0,0}$(1,0,0)$}%
}}}}
\end{picture}%
\end{center}
\caption{Exact data for the discrete conjugate net in Figure \ref{fig:CounterexampleInfRig0}. The numbers at the edges are the components of the vectors between adjacent vertices, with arrows indicating the directions.}
\label{fig:CounterexampleInfRig}
\end{figure}
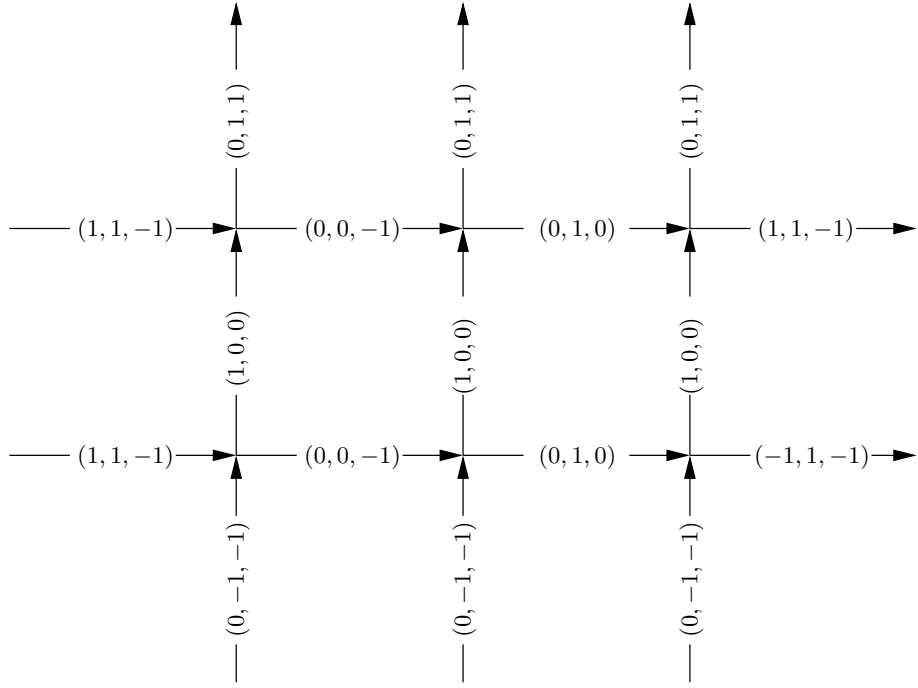

\begin{proof}
The principal feature of the net in Figure \ref{fig:CounterexampleInfRig} is that the dihedral angles at the top central and the bottom central edges vanish.
Figure \ref{fig:InfRigWeights} shows the weights assignment to the edges of each of the two $3 \times 3$-subnets satisfying the conditions \eqref{eqn:Weights}.
These assignments are unique up to scaling and do not agree on the central edges.
Therefore the $3 \times 4$ net is infinitesimally rigid.
\end{proof}

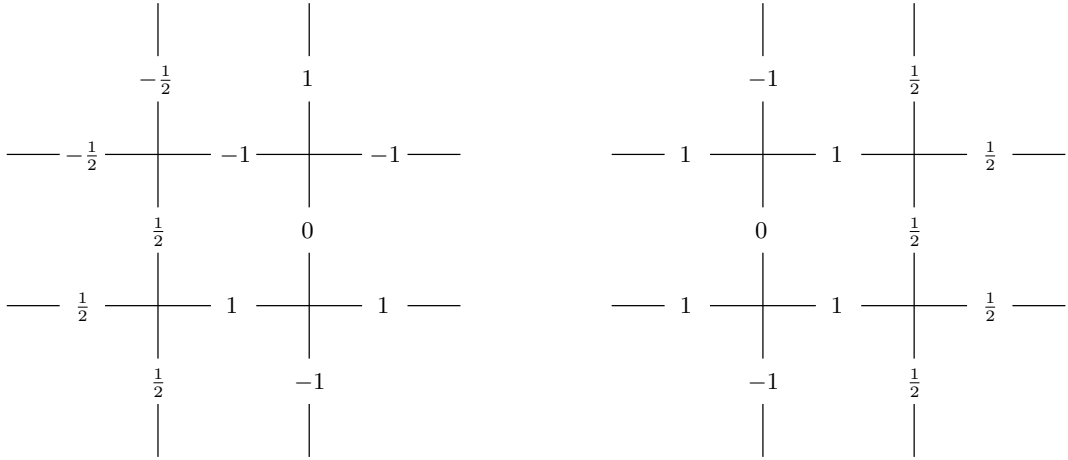
\begin{figure}[ht]
\begin{center}
\begin{picture}(0,0)%
\includegraphics{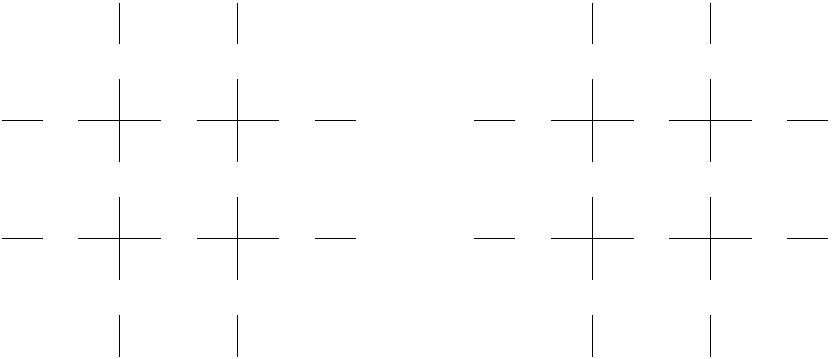}%
\end{picture}%
\setlength{\unitlength}{4144sp}%
\begin{picture}(6324,2724)(-11,-1873)
\put(2206,-1006){\makebox(0,0)[lb]{\smash{\fontsize{9}{10.8}\usefont{T1}{ptm}{m}{n}{\color[rgb]{0,0,0}$1$}%
}}}
\put(406,-1006){\makebox(0,0)[lb]{\smash{\fontsize{9}{10.8}\usefont{T1}{ptm}{m}{n}{\color[rgb]{0,0,0}$\frac12$}%
}}}
\put(1306,-1006){\makebox(0,0)[lb]{\smash{\fontsize{9}{10.8}\usefont{T1}{ptm}{m}{n}{\color[rgb]{0,0,0}$1$}%
}}}
\put(2161,-106){\makebox(0,0)[lb]{\smash{\fontsize{9}{10.8}\usefont{T1}{ptm}{m}{n}{\color[rgb]{0,0,0}$-1$}%
}}}
\put(856,-556){\makebox(0,0)[lb]{\smash{\fontsize{9}{10.8}\usefont{T1}{ptm}{m}{n}{\color[rgb]{0,0,0}$\frac12$}%
}}}
\put(856,-1456){\makebox(0,0)[lb]{\smash{\fontsize{9}{10.8}\usefont{T1}{ptm}{m}{n}{\color[rgb]{0,0,0}$\frac12$}%
}}}
\put(347,-106){\makebox(0,0)[lb]{\smash{\fontsize{9}{10.8}\usefont{T1}{ptm}{m}{n}{\color[rgb]{0,0,0}$-\frac12$}%
}}}
\put(1269,-106){\makebox(0,0)[lb]{\smash{\fontsize{9}{10.8}\usefont{T1}{ptm}{m}{n}{\color[rgb]{0,0,0}$-1$}%
}}}
\put(785,344){\makebox(0,0)[lb]{\smash{\fontsize{9}{10.8}\usefont{T1}{ptm}{m}{n}{\color[rgb]{0,0,0}$-\frac12$}%
}}}
\put(1756,344){\makebox(0,0)[lb]{\smash{\fontsize{9}{10.8}\usefont{T1}{ptm}{m}{n}{\color[rgb]{0,0,0}$1$}%
}}}
\put(1756,-556){\makebox(0,0)[lb]{\smash{\fontsize{9}{10.8}\usefont{T1}{ptm}{m}{n}{\color[rgb]{0,0,0}$0$}%
}}}
\put(1711,-1456){\makebox(0,0)[lb]{\smash{\fontsize{9}{10.8}\usefont{T1}{ptm}{m}{n}{\color[rgb]{0,0,0}$-1$}%
}}}
\put(4006,-1006){\makebox(0,0)[lb]{\smash{\fontsize{9}{10.8}\usefont{T1}{ptm}{m}{n}{\color[rgb]{0,0,0}$1$}%
}}}
\put(4906,-1006){\makebox(0,0)[lb]{\smash{\fontsize{9}{10.8}\usefont{T1}{ptm}{m}{n}{\color[rgb]{0,0,0}$1$}%
}}}
\put(4411,-1456){\makebox(0,0)[lb]{\smash{\fontsize{9}{10.8}\usefont{T1}{ptm}{m}{n}{\color[rgb]{0,0,0}$-1$}%
}}}
\put(4456,-556){\makebox(0,0)[lb]{\smash{\fontsize{9}{10.8}\usefont{T1}{ptm}{m}{n}{\color[rgb]{0,0,0}$0$}%
}}}
\put(5806,-1006){\makebox(0,0)[lb]{\smash{\fontsize{9}{10.8}\usefont{T1}{ptm}{m}{n}{\color[rgb]{0,0,0}$\frac12$}%
}}}
\put(5356,-556){\makebox(0,0)[lb]{\smash{\fontsize{9}{10.8}\usefont{T1}{ptm}{m}{n}{\color[rgb]{0,0,0}$\frac12$}%
}}}
\put(5356,-1456){\makebox(0,0)[lb]{\smash{\fontsize{9}{10.8}\usefont{T1}{ptm}{m}{n}{\color[rgb]{0,0,0}$\frac12$}%
}}}
\put(5356,344){\makebox(0,0)[lb]{\smash{\fontsize{9}{10.8}\usefont{T1}{ptm}{m}{n}{\color[rgb]{0,0,0}$\frac12$}%
}}}
\put(5806,-106){\makebox(0,0)[lb]{\smash{\fontsize{9}{10.8}\usefont{T1}{ptm}{m}{n}{\color[rgb]{0,0,0}$\frac12$}%
}}}
\put(4906,-106){\makebox(0,0)[lb]{\smash{\fontsize{9}{10.8}\usefont{T1}{ptm}{m}{n}{\color[rgb]{0,0,0}$1$}%
}}}
\put(4006,-106){\makebox(0,0)[lb]{\smash{\fontsize{9}{10.8}\usefont{T1}{ptm}{m}{n}{\color[rgb]{0,0,0}$1$}%
}}}
\put(4411,344){\makebox(0,0)[lb]{\smash{\fontsize{9}{10.8}\usefont{T1}{ptm}{m}{n}{\color[rgb]{0,0,0}$-1$}%
}}}
\end{picture}%
\end{center}
\caption{Weights assignments to the edges of $3 \times 3$-subnets of the net from Figure \ref{fig:CounterexampleInfRig}.}
\label{fig:InfRigWeights}
\end{figure}

\begin{rem}
Each of the $3 \times 3$-subnets of the net from Theorem \ref{thm:Counterexample3} is not only infinitesimally flexible, but also allows an isometric deformation.
The left hand side net is of translational type, the right hand side net is of symmetric type, see \cite{Stachel2010}.
Therefore this net provides a third counterexample to the globalization theorem of Schief--Bobenko--Hoffmann.
\end{rem}

\section{Further remarks}
\subsection{Flexibility over complex numbers}
Each of our three counterexamples has some distinctive features.

The first counterexample while being rigid over $\RR$ is flexible over $\CC$.
By this we mean the following.
The configuration space of a spherical quadrilateral (and thus that of a $2 \times 2$-complex of Euclidean quadrilaterals) becomes after the tangent half-angles substitution a real algebraic set.
The complexification of this algebraic set lifts the restriction on the range of the cosines in equations \eqref{eqn:CosineSystem}, and makes the discrete conjugate net ``flexible over complex numbers''.

The situation is different for the second counterexample.
We have used the tangent half-angles substitution in Section \ref{sec:ex2}, and found that the configuration space of each $2 \times 2$-subnet is a reducible algebraic set, see Lemma \ref{lem:ConfSpace}.
The configuration space of each of the two $3 \times 3$-subnets combines certain irreducible components of the configuration spaces of $2 \times 2$-subnets.
The rigidity stems from the fact that the left and the right $3 \times 3$-subnets select different components from the configuration spaces of the central $2 \times 2$-subnets.
Figure \ref{fig:Combinations} illustrates this.
As a result, the rigidity persists after complexification.

A net flexible over $\CC$ is necessarily infinitesimally flexible over $\RR$.
Indeed, the velocity field of a path in the complex configuration space (a smooth path exists due to the algebraicity) is a complex solution of the system \eqref{eqn:Weights}.
Since the system has real coefficients, the existence of a complex solution implies the existence of a real solution.
Therefore the first counterexample is infitesmially flexible over $\RR$.
As the second counterexample is also infitesmially flexible, the three counterexamples realize all possible combinations of (in)flexibility over $\CC$ and infinitesimal (in)flexibility over $\RR$.

\subsection{Analogies in planar kinematics}
The spherical linkage in Figure \ref{fig:2coupling} formed by the tangent images of the interior vertices reflects the structure of the first counterexample.
This linkage has an obvious counterpart in the Euclidean plane.
The coupled tangent images of the interior vertices of the other two counterexamples are more complicated and do not have obvious counterparts in the Euclidean plane.
There are, however, simple Euclidean planar linkages that exhibit properties similar to the second and the third of our counterexamples.

Let us start with a four-bar linkage with all bars of the same length.
We fix one of these bars and call it the base, the opposite bar will be called the coupler link, the remaining two bars will be called arms.
In the configurations where all joints are collinear, the motion of the linkage bifurcates: one branch corresponds to a circular translation of the coupler link and the other one to its pure rotation.
A four-bar linkage can also be viewed as a pin-jointed bar-plate framework, where the coupler plate and the base plate are linked over two arms.
This framework corresponds to the central $(3\times 2)$-subnets of the second and the third counterexample allowing a motion bifurcation.
The $(3\times 3)$-subnets correspond to finite flexible parallel mechanisms, which are obtained by connecting the coupler plate and the base plate by a further arm.
The graphs of these frameworks with an overconstrained mobility are given in Figure \ref{fig:graphs}.
Each of the additional arms supports only one of the two motion branches of the initial four-bar linkage.
As a consequence the parallel mechanism containing all four arms cannot move out from the bifurcation configuration as in the second and the third counterexamples.

In Figure \ref{fig:pms} the parallel manipulators which correspond to the graphs given in Figure \ref{fig:graphs} are illustrated.
The one in Figure \ref{fig:pms}a is infinitesimally flexible and thus similar to our second counterexample.
The one in Figure \ref{fig:pms}b is infinitesimally rigid and thus similar to our third counterexample.

\begin{figure}[ht]
\begin{center}
\begin{overpic}[height=70mm]{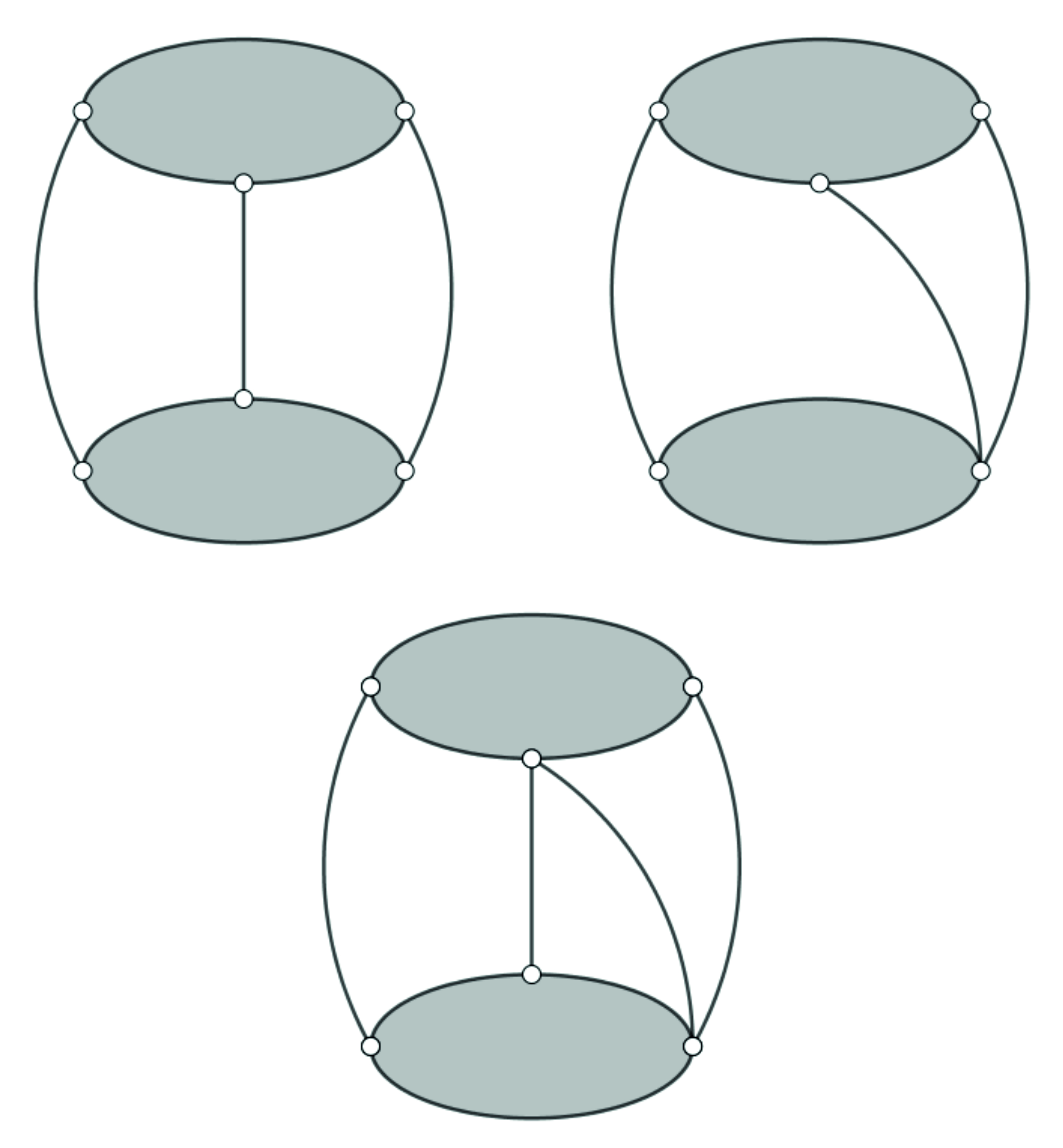}
  \begin{scriptsize}
    \put(0,0){a)}
    \put(12,65){$2$}
    \put(29,65){$2$}
    \put(12,81.5){$2$}
    \put(29,81.5){$2$}
     \put(0,73){$4$}
     \put(18,73){$4$}
     \put(40.5,73){$4$}
     \put(20,49){$4$}
     \put(20,98){$4$}
    \put(62,81.5){$2$}
    \put(79,81.5){$2$}
     \put(50,73){$4$}
     \put(77,73){$2$}
     \put(90.5,73){$4$}
     \put(70,49){$4$}
     \put(70,98){$4$}
      \put(37,15){$2$}
    \put(54,15){$2$}
    \put(37,31.5){$2$}
    \put(54,31.5){$2$}
     \put(25,23){$4$}
     \put(43,23){$4$}
     \put(65.5,23){$4$}
     \put(45,-1){$4$}
     \put(45,48){$4$}
      \put(52,23){$2$}
  \end{scriptsize}
\end{overpic}
\qquad \qquad
\begin{overpic}[height=70mm]{graphs_all}
  \begin{scriptsize}
    \put(0,0){b)}
     \put(12,65){$5$}
    \put(29,65){$3$}
    \put(12,81.5){$5$}
    \put(29,81.5){$3$}
     \put(0,73){$4$}
     \put(18,73){$4$}
     \put(40.5,73){$4$}
     \put(20,49){$4$}
     \put(20,98){$4$}
    \put(62,81.5){$5$}
    \put(79,81.5){$3$}
     \put(50,73){$4$}
     \put(77,73){$5$}
     \put(90.5,73){$4$}
     \put(70,49){$4$}
     \put(70,98){$4$}
      \put(37,15){$5$}
    \put(54,15){$3$}
    \put(37,31.5){$5$}
    \put(54,31.5){$3$}
     \put(25,23){$4$}
     \put(43,23){$4$}
     \put(65.5,23){$4$}
     \put(45,-1){$4$}
     \put(45,48){$4$}
      \put(52,23){$5$}
  \end{scriptsize}
\end{overpic}
\end{center}
\caption{Graphs of the pin-jointed bar-plate frameworks, where the plates are colored in gray, bars in black and the pin-joints in white. By assigning lengths to the bars, the inner geometry of the framework is determined. In (a) and (b) we consider two different geometries. The left graph corresponds to the parallel mechanism with the circular translation and the right graph to the one with the pure rotation. The central graph corresponds to the rigid parallel mechanism with four arms.  }
\label{fig:graphs}
\end{figure}

\begin{figure}[ht]
\begin{center}
\begin{overpic}[height=50mm]{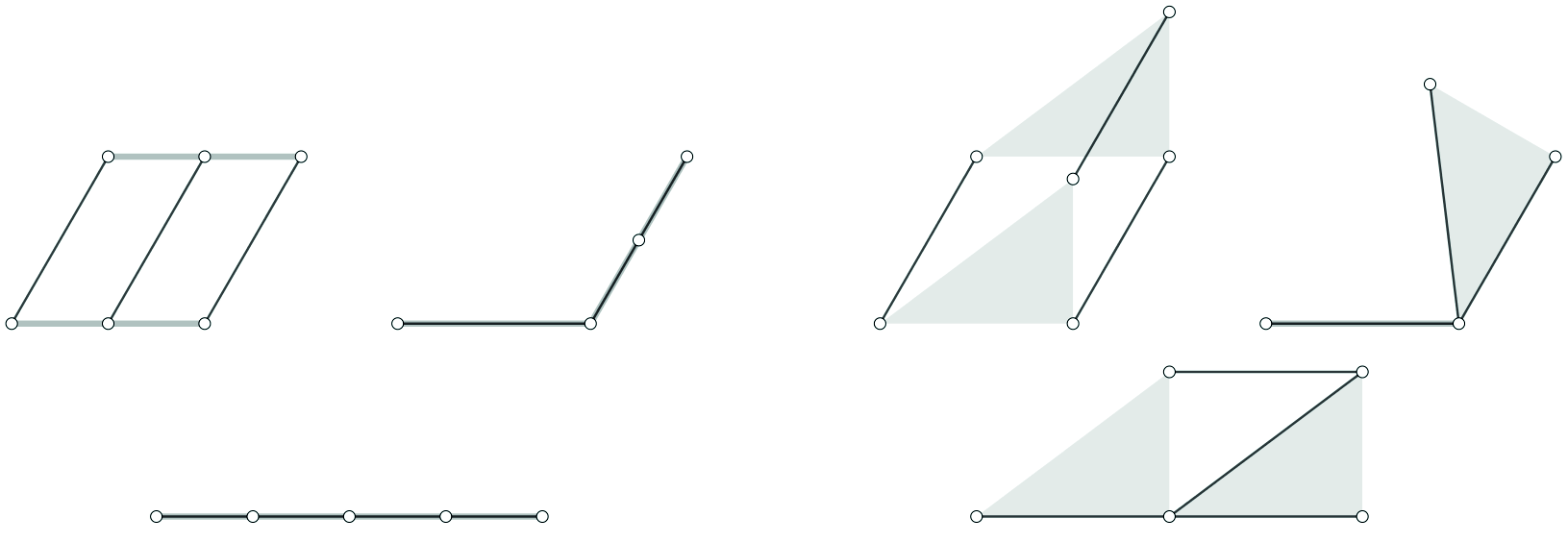}
  \begin{scriptsize}
    \put(0,0){a)}
    \put(50,0){b)}
  \end{scriptsize}
\end{overpic}
\end{center}
\caption{Parallel manipulators which correspond to the graphs given in Fig.\ \ref{fig:graphs}. The parallel mechanisms with four arms given in (a) is infinitesimal flexible but the one illustrated in (b) is first order rigid.
}
\label{fig:pms}
\end{figure}

\subsection{Non-simply connected discrete conjugate nets}
The globalization theorems \ref{thm:GlobFlex} and \ref{thm:GlobInfFlex} hold not only for rectangle-shaped discrete conjugate nets but for all simply-connected ones.
On the other hand, analysis and construction of flexible discrete conjugate nets of non-trivial topology requires methods different from those of \cite{SBH08, Stachel2010, Naw11, Naw12, Izmestiev2017}.
A special class of flexible tubes was studied in a recent work~\cite{Sharifmoghaddam2023}.

% \bibliographystyle{abbrv}
% \bibliography{Counterexample}

\end{document}